\documentclass{amsart}
\usepackage[utf8]{inputenc}
\usepackage[english]{babel}

\usepackage{pgfplots}
\pgfplotsset{compat=1.18}

\usepackage{csquotes}
\usepackage{enumitem}

\usepackage[backend=biber]{biblatex}
\usepackage{graphicx}
\usepackage{amsmath, amssymb, amsthm}

\newcommand{\R}{\mathbb R}

\newtheorem{prop}{Proposition}[section]
\newtheorem{theorem}{Theorem}[section]
\newtheorem{prob}{Problem}
\newtheorem{defn}{Definition}[section]
\newtheorem{remark}{Remark}[section]

\newtheorem{lemma}{Lemma}[section]
\newtheorem{example}{Example}[section]
\newtheorem{conj}{Conjecture}[section]
\newtheorem{cor}{Corollary}[section]
\newtheorem{construction}{Construction}

\newtheorem*{retheorem}{Theorem 1.2}

\DeclareMathOperator{\supp}{supp}

\DeclareFontFamily{U}{mathx}{}
\DeclareFontShape{U}{mathx}{m}{n}{<-> mathx10}{}
\DeclareSymbolFont{mathx}{U}{mathx}{m}{n}
\DeclareMathAccent{\widehat}{0}{mathx}{"70}
\DeclareMathAccent{\widecheck}{0}{mathx}{"71}

\usepackage[x11names]{xcolor}
\usepackage{hyperref}
\usepackage[capitalize]{cleveref}

\title[Weighted Fourier extension for the cone by circle tangencies]{A sharp weighted Fourier extension estimate for the cone in $\mathbb R^3$ based on circle tangencies}
\author{Alexander Ortiz}
\address{Alexander Ortiz, Rice University, Houston, TX, USA}
\email{ao80@rice.edu}

\begin{document}

\begin{abstract}
We apply recent circle tangency estimates due to Pramanik--Yang--Zahl to prove sharp weighted Fourier extension estimates for the cone in $\mathbb R^3$ and $1$-dimensional weights. The idea of using circle tangency estimates to study Fourier extension of the cone is originally due to Tom Wolff, who used it in part to prove the first decoupling estimates. We make an improvement to the best known Mizohata--Takeuchi-type estimates for the cone in $\R^3$ and the 1-dimensional weights as a corollary of our main theorem, where the previously best known bound follows as a corollary of refined decoupling estimates.
\end{abstract}

\maketitle

\tableofcontents

\section{Introduction}
\subsection{Weighted estimates and the Mizohata--Takeuchi conjecture} In $\R^n$, if $\mathcal M\subset B^n(0,1)$ is a compact manifold with smooth surface measure $d\sigma$, the Fourier extension operator for functions on $\mathcal M$ is defined by
\[
E_\mathcal Mf(x) = \widecheck{fd\sigma}(x) = \int_{\mathcal M}f(\xi)e^{2\pi i x\cdot\xi}\,d\sigma(\xi),\quad x\in \R^n.
\]
Since $\mathcal M\subset B^n(0,1)$ is compact, $E_{\mathcal M}f$ is approximately constant on unit balls of $\mathbb R^n$. In particular, the level sets
\[
U_\alpha = \{x \in \R^n: |E_{\mathcal M}f(x)| > \alpha\}
\]
are approximately disjoint unions of unit balls. See for example Sections 2--5 of \cite{bourgain2011bounds} for a standard formalization of this version of the ``locally constant'' property. If $R>1$ and $X\subset B_R$ is a disjoint union of unit balls in an arbitrary ball of radius $R$, we are interested in how much the level sets $U_\alpha$ can concentrate on the ``test set'' $X$.
\begin{prob}[$L^2$ sparse restriction problem]\label{prob:shape-estimate}
    If $X\subset B_R$ is an arbitrary disjoint union of unit balls, what is the best constant $S_2(\mathcal M,X)$ so that
    \[
    (\int_X|E_{\mathcal M}f|^2\,dx)^{1/2} \le S_2(\mathcal M,X)\|f\|_{L^2(\mathcal M,d\sigma)}
    \]
    holds for all $f\in L^2(\mathcal M,d\sigma)$?
\end{prob}
If $\mathcal M$ is a hypersurface and $X\subset B_R$, by the conservation of mass,
\[
1 \le S_2(\mathcal M,X) \le R^{1/2},
\]
and examples of particular $\mathcal M$ and $X$ show that the shape of $X$ as well as the shape of $\mathcal M$ influence the size of the constant $S_2(\mathcal M,X)$. As a basic example of this phenomenon, consider the vertical column $X = [0,1]\times [0,R]\subset\R^2$, and the manifolds $\mathcal M_1 = \{(\xi,0):\xi\in[0,1]\}$, $\mathcal M_2 = \{(\xi,\xi^2):\xi\in[0,1]\}$. Simple calculations with the Fourier transform show $S_2(\mathcal M_1,X) \sim R^{1/2}$, while $S_2(\mathcal M_2,X)\sim R^{1/4}$.

Weighted Fourier extension estimates for model manifolds such as the paraboloid or the cone have applications to dispersive PDE \cite{barcelo1997weighted} and to geometric problems such as Falconer's distance problem. See the influential paper by Du and Zhang \cite{du2019sharp} for a weighted Fourier extension theorem and its application to pointwise convergence to the initial data for the Schr\"odinger equation, and the paper \cite{du2021weighted} for partial progress towards Falconer's distance set problem as a corollary of weighted Fourier extension estimates for spheres.

Problem \ref{prob:shape-estimate} can be seen as a special case of the Mizohata--Takeuchi conjecture of Fourier analysis, where the weight is the indicator function of a disjoint union of unit balls. Let $\mathcal M$ be a smooth compact hypersurface in $\R^n$, $n\ge 2$.  If $T$ is a tube in $\R^n$, let $\nu(T)$ be one of the two unit vectors parallel to the central axis of $T$. If a vector $v$ is orthogonal to $T_\xi\mathcal M$ for some point $\xi\in \mathcal M$, we write $v\perp \mathcal M$.
\begin{conj}[Local Mizohata--Takeuchi]\label{conj:mt-intro}
For every $\epsilon>0$, there is a constant $C_\epsilon$ such that the following holds for every $R>1$. If $\mathcal M\subset\mathbb R^n$ is a smooth compact hypersurface, $n\ge 2$, let
\[
\mathbb T(\mathcal M) = \{T\subset\R^n: \nu(T)\perp \mathcal M,\ T \ \text{is a $1\times \dots\times 1\times R$-tube}\} 
\]
be the collection of $1\times\dots\times 1\times R$-tubes whose direction is orthogonal to a tangent space of $\mathcal M$. If $w\colon \R^n\to[0,\infty)$ is a measurable weight, let
\[
\mathbf T_{\mathcal M}(w) = \sup\{\int_T w:T\in \mathbb T(\mathcal M)\}.
\]
Then the following estimate holds for all $f\in L^2(\mathcal M)$ and any $R$-ball $B_R\subset\R^n$:
\begin{equation}\label{eqn:mt-inequality}
    \int_{B_R}|E_{\mathcal M}f|^2w \le C_\epsilon R^\epsilon\,\mathbf T_{\mathcal M}(w)\|f\|_{L^2(\mathcal M)}^2.
\end{equation}
\end{conj}

Unlike the Fourier restriction conjecture, Conjecture \ref{conj:mt-intro} is made without any assumption on the curvature of $\mathcal M$. When the weight is constant on translates of a fixed hyperplane, and the hypersurface $\mathcal M$ is arbitrary, the conjecture is true essentially by Plancherel's theorem. When $\mathcal M$ is a sphere in $\R^n$, the conjecture is open in all dimensions $n\ge 2$. In the case $n = 2$, it is known for radial weights---see \cite{barcelo1997weighted,carbery1997pointwise,barcelo2008note}. Conjecture \ref{conj:mt-intro} has its roots in questions about dispersive PDE \cite{mizohata2014cauchy}, but it also has a close connection with the endpoint multilinear restriction conjecture, as well as with the Bochner--Riesz problem. See the article \cite{bennett2006stein} and the references therein for more discussion on Mizohata--Takeuchi-type estimates for spheres and the connection with the Bochner--Riesz problem on spherical Fourier summation.

Conjecture \ref{conj:mt-intro} is a local form of the Mizohata--Takeuchi conjecture with an $R^\epsilon$ loss. There is a global X-ray transform formulation of the Mizohata--Takeuchi conjecture which has received much attention from the Fourier analysis community in recent years, culminating in its disproof by Hannah Cairo \cite{cairo2025counterexample}.
\begin{conj}[Global Mizohata--Takeuchi]\label{conj:global-mt}
    There is an absolute constant $C>0$ such that the following holds.
    Let $\mathcal M$ be a smooth compact  hypersurface in $\R^n$ with surface measure $d\sigma$. Let $f\in L^2(\mathcal M,d\sigma)$, and let $w\colon\R^n\to[0,\infty)$ be a nonnegative measurable weight. Then we have
    \[
    \int_{\R^n}|E_{\mathcal M}f|^2w \le C\, \|Xw\|_{L^\infty}\|f\|_{L^2(d\sigma)}^2,
    \]
    where $Xw$ denotes the X-ray transform of $w$.
\end{conj}
Cairo's counterexample to Conjecture \ref{conj:global-mt} shows that in the local Mizohata--Takueuchi Conjecture \ref{conj:mt-intro}, the loss in the inequality must be at least of size $\log R$, so it does not rule out the possibility that Conjecture \ref{conj:mt-intro} holds.

In Section \ref{sec:discuss} we review the state of the art on Conjecture \ref{conj:mt-intro} for spheres due to Carbery, Iliopoulou, and Wang based on refined decoupling estimates \cite{carbery2024some}. We compare the strength of our main Theorem \ref{thm:main-l2-sparse} for the cone and $1$-dimensional weights with Theorem \ref{thm:ciprian} due to Ciprian Demeter.

As the Mizohata--Takeuchi conjecture is expected to hold regardless of the curvature of the manifold $\mathcal M$, we expect it to be fruitful for our overall understanding to see what we can say about weighted Fourier extension estimates for model manifolds of zero Gaussian curvature. In this paper, we investigate $L^2(w)$ estimates for $E_{\mathbb Cone^2}f$, where
\[
\mathbb Cone^2 = \{(\bar\xi,\xi_3)\in\mathbb R^2\times\R:1<|\bar\xi|<2,\xi_3=|\bar\xi|\}
\]
is the truncated cone in $\R^3$, and $w$ is the indicator function of a 1-dimensional disjoint union of unit balls. To state our main theorem, we make a preliminary definition. Let $\gamma(\theta) = \frac{1}{\sqrt 2}(\cos\theta,\sin\theta,1),\theta\in\R$. 
\begin{defn}
    If for some $\theta\in\R$ and $v\in\mathbb R^3$,
    \[
    P = v + \{a\gamma(\theta) + b\gamma'(\theta) + c(\gamma\times\gamma')(\theta): |a|\le \frac{R}{2},|b|\le \frac{R^{1/2}}{2},|c|\le \frac12\},
    \]
    then we say $P$ is a \emph{$1\times R^{1/2}\times R$-lightplank} with \emph{center} $v$.
\end{defn}
\begin{theorem}\label{thm:main-l2-sparse}
    For each $\epsilon > 0$,  there is a constant $C_\epsilon$ so the following holds for each $R > 1$. Suppose $X\subset B_R$ is a disjoint union of unit balls in $\R^3$ that satisfies the 1-dimensional Frostman non-concentration condition
    \begin{equation}\label{eqn:1-dimensional}
        |X\cap B(x,r)|\lesssim  r,\quad x\in \R^3,r>1.
    \end{equation}
    Let $\mathbf P(X)$ be the quantity
    \[
    \mathbf P(X) = \sup\{|X\cap P|:P\ \text{is a $1\times R^{1/2}\times R$-lightplank}\}.
    \]
    Then the estimate
    \begin{equation}\label{eqn:main-l2-sparse}
    \int_X |E_{\mathbb Cone^2}f|^2\le C_\epsilon R^\epsilon\, \mathbf P(X)^{1/2}\|f\|_{L^2(\mathbb Cone^2)}^2
    \end{equation}
    holds for all $f\in L^2(\mathbb Cone^2)$.
\end{theorem}
When a disjoint union of unit balls $X$ satisfies \eqref{eqn:1-dimensional}, we will say $X$ is \emph{$1$-dimensional}. The terminology behind referring to \eqref{eqn:1-dimensional} as a ``Frostman condition'' comes from geometric measure theory, where measures $\mu$ satisfying the power-law $\mu(B_r) \lesssim r^\alpha$ are known as \emph{Frostman measures of exponent $\alpha$}. The weighted estimate of Theorem \ref{thm:main-l2-sparse} is sharp in the sense that for each $R>1$, and for each $1<T<R$, there is a 1-dimensional disjoint union of unit balls $X$ such that $\mathbf P(X) \sim T$, and a nonzero function $f$ such that
\[
\int_X |E_{\mathbb Cone^2}f|^2 \gtrsim T^{1/2}\|f\|_{L^2(\mathbb Cone^2)}^2.
\]
Theorem \ref{thm:main-l2-sparse} is a corollary of the following weighted $L^1$ Fourier extension estimate.
\begin{theorem}\label{thm:ortiz-sparse-l1}
    For each $\epsilon > 0$,  there is a constant $C_\epsilon$ so the following holds for each $R > 1$. Suppose $X\subset B_R$ is a $1$-dimensional disjoint union of unit balls.
    
    Let $\mathbf P(X)$ be the quantity
    \[
    \mathbf P(X) = \sup\{|X\cap P|:P\ \text{is a $1\times R^{1/2}\times R$-lightplank}\}.
    \]
    Then for any measurable function $w\colon\R^3\to[0,1]$, the estimate
    \[
    \int_X|E_{\mathbb Cone^2}f|w\le C_\epsilon R^\epsilon\,\mathbf P(X)^{1/4}(\int_Xw)^{1/2}\|f\|_{L^2(\mathbb Cone^2)}
    \]
    holds for all $f\in L^2(\mathbb Cone^2)$.
\end{theorem}
We give the proof of Theorem \ref{thm:ortiz-sparse-l1} and the implication that Theorem \ref{thm:ortiz-sparse-l1} implies Theorem \ref{thm:main-l2-sparse} in Section \ref{sec:main-thm}.  As a corollary of Theorem \ref{thm:main-l2-sparse}, we have the following new partial progress on the Mizohata--Takeuchi conjecture for the cone and the $1$-dimensional weights.

\begin{cor}\label{cor:mt-progress}
    For each $\epsilon>0$, there is a constant $C_\epsilon$ so that if $X$ is a 1-dimensional disjoint union of unit balls in $B_R\subset\R^3$, then
    \[
    \int_X|E_{\mathbb Cone^2}f|^2 \le C_\epsilon R^{1/4+\epsilon}\,\mathbf T_{\mathbb Cone^2}(X)\|f\|_{L^2(\mathbb Cone^2)}^2.
    \]
\end{cor}
Corollary \ref{cor:mt-progress} establishes the case of Conjecture \ref{conj:mt-intro} for the cone and the 1-dimensional weights with a power $R^{1/4}$ loss. This is an improvement of the $R^{1/3}$ loss we get by applying  Theorem \ref{thm:ciprian} directly to $1$-dimensional weights.

Our approach to estimating $\int_X|Ef|^2$ is based on a duality argument originally due to Mattila \cite{mattila1987spherical} which connects weighted Fourier extension estimates with the decay of Fourier means, and the technique of point-circle duality which was originally used by Wolff in \cite{wolff2000local} to study $E_{\mathbb Cone^2}f$. 

\subsection{Decay of Fourier means}\label{sec:decay-fourier} Weighted Fourier extension estimates are closely related to average Fourier decay rates of measures, and our methods also give new results on average Fourier decay. Given a finite measure $\mu$ on $\R^n$, we define its Fourier transform by
\[
\widehat\mu(\xi) = \int_{\R^n}e^{-2\pi ix\cdot\xi}\,d\mu(x),\quad \xi\in\R^n.
\]
If $d\mu = \phi(x)\,dx$ for a Schwartz function $\phi$, then $\widehat\mu(\xi)$ decays rapidly as $|\xi|\to\infty$. On the other hand, without any assumptions on $\mu$, $\widehat\mu(\xi)$ need not decay at all as $|\xi|\to\infty$ because we allow very singular measures such as  $\mu = \delta_0$ where $\widehat\mu = 1$ everywhere. An example of a measure which fits in-between these two extremes is the product measure
\[
d\mu(x_1,x_2) = \phi(x_1)\,dx_1\otimes \delta_0(dx_2),
\]
where $\phi$ is a smooth compactly supported function on $\R$. The measure $\widehat\mu(\xi_1,\xi_2) = \widehat\phi(\xi_1)$ decays rapidly in the $\xi_1$ variable, but is constant in the $\xi_2$ variable. This motivates the definition of measures with finite $\alpha$-energy.

\begin{defn}
    If $\mu$ is a measure on $\R^n$, and $0<\alpha<n$, define the \emph{$\alpha$-energy}
    \[
    I_\alpha(\mu) = \iint \frac{d\mu(x)d\mu(y)}{|x-y|^\alpha}.
    \]
\end{defn}

In 1987, Mattila \cite{mattila1987spherical} introduced the problem of determining the optimal decay rate of spherical Fourier averages of compactly supported measures with finite $\alpha$-energy. This was proposed as a Fourier analytic approach to Falconer's distance set conjecture.

\begin{defn}[Optimal spherical Fourier average decay rate]
    For $n\ge 2$ and $0<\alpha<n$, $\beta_n(\alpha)$ is the supremum of the numbers $\beta>0$ such that there exists $C$ with
    \[
    \int_{S^{n-1}}|\widehat\mu(Re)|^2\,d\sigma(e)\le CR^{-\beta}I_\alpha(\mu),\quad\text{for all $R>1$,}
    \]
    for all measures $\mu$ supported in the unit ball of $\R^n$.
\end{defn}

In \cite{mattila1987spherical}, Mattila established the sharp exponent $\beta_n(\alpha) = \alpha$ when $0<\alpha\le \frac{n-1}{2}$. For the lower bound $\beta_n(\alpha)\ge \alpha$, Mattila used the stationary phase estimate $|\widecheck{d\sigma}(x)|\lesssim|x|^{-(n-1)/2}$, where $d\sigma$ is the surface measure of the sphere, by rewriting the Fourier average as a double integral
\begin{equation}\label{eqn:mattila}
\int_{S^{n-1}}|\widehat\mu(Re)|^2\,d\sigma(e) = \iint_{\R^n\times\R^n} \widecheck{d\sigma}(x-y)\,d\mathrm{Dil}_R\mu(x)d\mathrm{Dil}_R\mu(y),
\end{equation}
where $(\mathrm{Dil}_R\mu)(x) = R^{-n}\mu(\frac{x}{R})$.
Mattila gave examples of measures with finite $\alpha$-energy to prove the upper bound $\beta_n(\alpha)\le \alpha$ when $0<\alpha\le\frac{n-1}{2}$.

Following partial progress on the decay of spherical Fourier averages by Mattila \cite{mattila1987spherical}, Sj\"olin \cite{sjolin1993estimates}, and Wolff \cite{wolff1999decay}---who determined the optimal values of \(\beta_2(\alpha)\) for \(\alpha \in (1,2)\)---and advances in the Fourier restriction problem by Tao \cite{tao2003sharp}, Erdo\u{g}an \cite{erdogan2004note} established the sharp decay rates for \emph{conical} Fourier averages of \(\alpha\)-dimensional measures in \(\mathbb{R}^3\). These conical averages exemplify a broader class of Fourier averages over submanifolds that are linked to various problems in geometric measure theory and dispersive PDEs. In particular, they are equivalent to \(L^2\) fractal Strichartz estimates for the wave equation \cite{cho2017fractal,wolff2000local} and have applications to Marstrand-type theorems for restricted families of projections \cite{iosevich2014decay,oberlin2015application}. We now make precise the notion of the optimal conical Fourier decay rate:

\begin{defn}[Optimal conical Fourier average decay rate]\label{def:cone-rate}
    For $n\ge 3$, let
    \[
    \mathbb Cone^{n-1} = \{(\bar \xi,\xi_n)\in \R^{n-1}\times\R:1<|\bar \xi|<2,\xi_n = |\bar\xi|\}.
    \]
    For $0<\alpha<n$, $\gamma_n(\alpha)$ is the supremum of the numbers $\gamma>0$ such that there exists $C$ with
    \[
    \int_{\mathbb Cone^{n-1}}|\widehat\mu(Re)|^2\,d\sigma(e)\le CR^{-\gamma}I_\alpha(\mu),\quad\text{for all $R>1$,}
    \]
    for all measures $\mu$ supported in the unit ball of $\R^n$.
\end{defn}
\begin{theorem}[Erdo\u gan, 2004]\label{thm:erdogan}
For $\alpha \in (0,3)$, the values of $\gamma_3(\alpha)$ are as follows:
\[
 \gamma_3(\alpha) = 
\begin{cases}
    \alpha, & \alpha\in (0,1/2]\\
    1/2, & \alpha \in [1/2,1]\\
    \alpha/2, & \alpha \in [1,2]\\
    \alpha-1, & \alpha \in [2,3)
\end{cases}.
\]
\end{theorem}

\begin{figure}[h]
\centering
\begin{tikzpicture}
\begin{axis}[
    axis lines = left,
    xlabel = $\alpha$,
    ylabel = {},
    ymin = 0, ymax = 2,
    xmin = 0, xmax = 3,
    domain=0:3,
    samples=100,
    xtick={0,0.5,1,2,3},
    ytick={0,0.5,1,1.5,2},
    grid=both,
    width=12cm,
    height=8cm
]

\addplot[domain=0:0.5, thick, blue] {x};

\addplot[domain=0.5:1, thick, blue] {0.5};

\addplot[domain=1:2, thick, blue] {x/2};

\addplot[domain=2:2.99, thick, blue] {x - 1};

\addplot[only marks, mark=*, mark size=2pt] coordinates {(0.5,0.5) (1,0.5) (2,1)};
\addplot[only marks, mark=o, mark size=2pt] coordinates {(0.5,0.5) (2,1)};

\end{axis}
\end{tikzpicture}
\caption{Plot of $\gamma_3(\alpha)$}
\end{figure}

A special case of our main $L^1$ weighted Fourier extension estimate contains a refinement of the $\alpha = 1$ case of Erdo\u gan's estimate. In the rest of the introduction, we will focus on this special case whose proof contains most of the key ideas needed to prove Theorem \ref{thm:ortiz-sparse-l1}, and in Section \ref{sec:main-thm} we will elaborate more on the connection between weighted Fourier extension estimates and the decay of Fourier means.
\begin{theorem}[Refinement of $\gamma_3(1)=1/2$]\label{thm:main-fourier-decay}
For each $\epsilon > 0$, there is a constant $C_\epsilon$ so the following holds for each $R > 1$. Suppose $X\subset B_R$ is a 1-dimensional disjoint union of unit balls, and let $d\mu = 1_Xdx$.
Let $\mathbf P(\mu)$ be the quantity
\[
\mathbf P(\mu) = \sup\{\mu(P):P\ \text{is a $1\times R^{1/2}\times R$-lightplank}\}.
\]
Then the estimate
\[
\int_{\mathbb Cone^{2}} |\widehat\mu(e)|^2d\sigma(e) \le C_\epsilon R^{\epsilon}\, \mathbf P(\mu)^{1/2}\mu(B_R)
\]
holds.
\end{theorem}
The assumption \( d\mu = 1_X\,dx \), where \( X \) is a 1-dimensional disjoint union of unit balls, in Theorem~\ref{thm:main-fourier-decay} differs slightly from the hypothesis in Theorem~\ref{thm:erdogan}, which considers measures supported in the unit ball with finite \( 1 \)-energy. Although finite \( 1 \)-energy is not precisely equivalent to satisfying a \( 1 \)-dimensional Frostman condition, the two are closely related. More importantly, the two contexts can be brought into alignment via a simple rescaling: applying Theorem~\ref{thm:erdogan} to a pushforward of \( \mu \) under dilation yields an estimate of the quantity
\[
\int_{S^{n-1}} |\widehat{\mu}(e)|^2\, d\sigma(e),
\]
for measures supported in balls of arbitrary radius. With this observation, one can easily verify that Theorem~\ref{thm:erdogan} implies the following slightly weaker version of Theorem~\ref{thm:main-fourier-decay}:

\begin{cor}\label{cor:erdogan}
    For each $\epsilon > 0$, there is a constant $C_\epsilon$ so the following holds for each $R > 1$. Suppose $X\subset B_R$ is a 1-dimensional disjoint union of unit balls, and let $d\mu = 1_Xdx$. Then the estimate
    \[
    \int_{\mathbb Cone^2} |\widehat\mu(e)|^2d\sigma(e) \le C_\epsilon R^\epsilon R^{3/2}
    \]
    holds.
\end{cor}

If $d\mu = 1_Xdx$ is a measure satisfying the assumptions of Theorem \ref{thm:main-fourier-decay}, then $\mathbf P(\mu)\le \mu(B_R) \lesssim R$, so Theorem \ref{thm:main-fourier-decay} also  implies Corollary \ref{cor:erdogan}. However, for measures $\mu$ where $\mathbf P(\mu)$ is much smaller than $R$, Theorem \ref{thm:main-fourier-decay} gives a better estimate than Corollary \ref{cor:erdogan}.

\subsection{Circle tangencies and an overview of the proof of Theorem \ref{thm:main-fourier-decay}} Our approach to conical Fourier average decay is based on Equation \eqref{eqn:mattila} with $d\mu = 1_Xdx$. We rewrite the conical Fourier mean of $\mu$ as a double integral:
\begin{equation}\label{eqn:proof-start}
\int |\widehat\mu|^2d\sigma = \iint_{X\times X} \widecheck{d\sigma}(x-y)\,dx\,dy
\end{equation}
and we note the following fact about the cone. If $d\sigma$ is a smooth surface carried measure for $\mathbb Cone^2$, and $\Gamma_0 = \{(a,r)\in\R^2\times\R:||a|-|r||=0\}$ is the lightcone with vertex $0$, then for every $\epsilon>0$ and $N>1$,
\begin{equation}\label{eqn:cone-decay}
|\widecheck{d\sigma}(x)|\le C(\epsilon,N)\frac{1}{(1+|x|)^{1/2-\epsilon}}\frac{1}{(1+d(x,\Gamma_0))^N}.
\end{equation}
The inequality \eqref{eqn:cone-decay} is not a new estimate---see for instance expression (13) with $n = 3,\nu=1$ in the proof of Theorem 3 in Guo's article \cite{guo1993p}. To keep this article self-contained, we include the proof by stationary phase considerations in Appendix \ref{app:ft}.

Heuristically by inequality \eqref{eqn:cone-decay}, the only pairs $(x,x')\in X^2$ which contribute to the double integral in Equation \eqref{eqn:proof-start} are those such that $x-x'$ is close to the lightcone $\Gamma_0$. By the rapid decay in inequality \eqref{eqn:cone-decay}, the only pairs $x = (\bar x,x_3), x' = (\bar x',x_3')$ which contribute to the double integral satisfy
\[
\Delta(x,x') = ||\bar x-\bar x'|-|x_3-x_3'|| < 1.
\]
If we interpret the points $x=(\bar x,x_3)$ and $x'=(\bar x',x_3')$ of $X$ as circles in the plane with centers $\bar x,\bar x'$ and radii $x_3,x_3'$, respectively, then $\Delta(x,x')<1$ has the interpretation that the ``circles'' $x$ and $x'$ are almost internally tangent. We thus need a solution of a discretized variant (see Problem \ref{prob:pairs} and Theorem \ref{thm:num-ll-pairs}) of the following discrete tangency counting problem introduced by Wolff in \cite{wolff1999recent}.
\begin{prob}[Tangency counting problem]\label{prob:wolff-tang}
    If $X\subset \R^3$ is a finite set of points in $[-\frac{1}{5},\frac{1}{5}]^2\times[\frac{4}{5},\frac{6}{5}]$ that we regard as a collection of circles in the plane with given center-radius pairs, estimate the number of internal tangencies of the configuration $X$:
    \[
    CT(X) = \#\{(x,x')\in X^2: ||\bar x-\bar x'| - |x_3 - x_3'|| =0\}.
    \]
    Here we write a point $x\in\R^3$ in coordinates as $x = (\bar x,x_3)\in \R^2\times \R$.
\end{prob}
To address the discretized tangency counting problem in Problem~\ref{prob:wolff-tang}, we apply a powerful recent result of Pramanik--Yang--Zahl~\cite{pramanik2022furstenberg} as a black box. Their estimate can be viewed as a \emph{dual circular maximal function} estimate: whereas classical circular maximal function bounds control the size of maximal averages of a function $f$ over a family of circles satisfying geometric constraints, the dual perspective instead controls how densely such a family of circles can overlap at points in the plane.  This overlap control will be the key ingredient in bounding the number of pairwise internal tangencies among such circles.

We use the notation $E(z) = 1_E(z)$ for the indicator function of a set $E$ in the statement of the following theorem and elsewhere in the paper.
\begin{theorem}[Pramanik--Yang--Zahl, 2022]
    For each $\epsilon>0$, there is a constant $C_\epsilon$ so the following holds for all $\delta>0$ sufficiently small. Suppose $X\subset [-\frac15,\frac15]^2\times[\frac45,\frac65]$ is a set of $\delta$-separated points regarded as circles in the plane obeying the 1-dimensional Frostman non-concentration condition
    \[
    |X\cap B(x,r)|\lesssim \frac{r}{\delta},\quad x\in \R^3,r>\delta.
    \]
    Then the following estimate holds:
    \begin{equation}\label{eqn:pyz}
    \int_{\R^2}\big(\sum_{x\in X}C_{\delta,x}(z)\big)^{3/2}\,dz \le C_\epsilon \delta^{-\epsilon}\delta|X|,
    \end{equation}
    where $C_{\delta,x} = \{z\in\R^2:||z-\bar x|-x_3|<\delta\}$ is a $\delta$-thick annulus.
\end{theorem}
To apply Pramanik--Yang--Zahl's estimate successfully, we will need to set up an appropriate dictionary translating circle tangency estimates into a bound for the double integral in Equation \eqref{eqn:proof-start}. We take up this work in Section \ref{sec:pc-duality}.

The idea of applying circle tangencies to study $E_{\mathbb Cone^2}f$ is not new. In ``Local smoothing type estimates in $L^p$ for large $p$'' \cite{wolff2000local}, Wolff used the point-circle duality idea to study $E_{\mathbb Cone^2}f$. Part of his goal was to understand the large level sets of $E_{\mathbb Cone^2}f$, and he used estimates for the number of circle tangencies of configurations of circles as a key ingredient to prove the first sharp decoupling estimates for the cone in $L^p$ for large $p$.

To apply facts about circles to the Fourier extension of the cone successfully, Wolff also needed to partly describe a dictionary of lemmas that relates the geometry of circle tangencies in the plane to the geometry of points in the upper half space. Part of our goal in this paper is to establish the facts from this dictionary we need in order to prove Theorem \ref{thm:ortiz-sparse-l1}. Another goal of this presentation is to clarify and expand on some of the same facts Wolff used in \cite{wolff2000local}.

\subsection{Paper organization} In Section \ref{sec:geometry-rect-lightplanks}, we develop the correspondence between approximate circle tangencies and point-lightplank incidences, which we use to establish Theorem \ref{thm:num-ll-pairs}—a crucial component in the proof of Theorem \ref{thm:main-l2-sparse}. Section \ref{sec:main-thm} presents our main results on sparse Fourier restriction, including Theorems \ref{thm:ortiz-sparse-l1}, \ref{thm:main-l2-sparse}, and the conical Fourier average estimate Theorem \ref{thm:main-fourier-decay}. Finally, in Section \ref{sec:discuss}, we discuss how Theorem \ref{thm:main-l2-sparse} delivers improved partial progress on the local Mizohata–Takeuchi conjecture, going beyond what current approaches based on refined decoupling can achieve.

\subsection*{Acknowledgments}
I would like to acknowledge Larry Guth for his constant support and invaluable discussions throughout this project. I would also like to acknowledge Jill Pipher and Tainara Borges for insightful questions about the proof of Theorem \ref{thm:ortiz-sparse-l1} leading to a cleaner exposition of the implication ``Theorem \ref{thm:ortiz-sparse-l1} $\implies$ Theorem \ref{thm:main-l2-sparse}.'' Thanks to Yixuan Pang for drawing my attention to the reference of the Fourier decay estimate \eqref{eqn:cone-decay} from Guo's paper \cite{guo1993p}. I extend my appreciation to Ciprian Demeter who let me include the statement of Theorem \ref{thm:ciprian} here and in my doctoral thesis. Finally, thanks to the anonymous referees for their detailed and thoughtful feedback on earlier versions of this paper, and especially for drawing my attention to Wolff's use of the point-circle duality technique in \cite{wolff2000local}.

\section{Point-circle and rectangle-lightplank duality} \label{sec:pc-duality}

\label{sec:maximal-estimate}
Fix $\delta > 0$ and, for $a\in \R^2$ and $r\in [1,2]$, let $C_{\delta,a,r} = \{x\in \R^2:r-\delta < |x-a|<r+\delta\}$. If $f\colon \R^2\to \R$, then we define Wolff's circular maximal function $M_\delta f\colon[1,2]\to\R$ by
\[
M_\delta f(r) = \sup_{a\in\R^2}\frac{1}{|C_{\delta,a,r}|}\int_{C_{\delta,a,r}}|f(z)|\,dz.
\]
Originally in \cite{wolff1997kakeya}, Wolff proved the following estimate for the maximal function $M_\delta f$.
\begin{theorem}\label{thm:wolff-kakeya-circles}
If $\epsilon > 0$ then there is a constant $A_\epsilon$ such that for all $\delta > 0$ and $f$,
\begin{equation}\label{wolffL3}
\|M_\delta f\|_{L^3([1,2],dr)} \le A_\epsilon \delta^{-\epsilon}\|f\|_{L^3(\R^2,dz)}.
\end{equation}
\end{theorem}
The estimate \eqref{wolffL3} has an equivalent dual form. Suppose that $a(r)$ is a measurable choice of center for a circle in the plane of radius $r$, and $w(r)$ is a nonnegative weight function. Define a multiplicity function
\[
m[w,a](z) = \int_{1}^2 w(r)\frac{C_{\delta, a(r), r}(z)}{|C_{\delta, a(r), r}|}\,dr,\quad z\in \mathbb C.
\]
\begin{prop}[Multiplicity formulation of the maximal estimate]\label{prop:dual-wolff}
If $\epsilon > 0$ then there is a constant $A_\epsilon$ such that for all $\delta > 0$, $a(r)$ and $w(r)$,
\begin{equation}\label{dualWolff}
\|m[w,a]\|_{L^{3/2}(\R^2,dz)}\le A_\epsilon\delta^{-\epsilon}\|w\|_{L^{3/2}([1,2],dr)}.
\end{equation}
\end{prop}
\begin{prop}
Wolff's maximal estimate is equivalent to its dual formulation in terms of multiplicity functions.
\end{prop}
\begin{proof}
Suppose that \eqref{wolffL3} holds. By duality, for an appropriate $f\in L^3(\R^2,dz)$ with $\|f\|_3 = 1$,
\begin{align*}
\|m\|_{L^{3/2}(\R^2,dz)} &= \int_{\R^2} m(z)f(z)\,dz \\
&= \int_{\R^2}\bigg(\int_1^2 w(r)\frac{C_{\delta, a(r), r}(z)}{|C_{\delta, a(r), r}|}\,dr\bigg)f(z)\,dz \\
&= \int_1^2 w(r) \bigg(\frac{1}{|C_{\delta,a(r),r}|}\int_{C_{\delta,a(r),r}}f(z)\,dz\bigg)\,dr \\
&\le \int_1^2 w(r) M_\delta f(r)\,dr \\
&\le \|w\|_{L^{3/2}([1,2],dr)}\|M_\delta f\|_{L^{3}([1,2],dr)} \\
&\le A_\epsilon \delta^{-\epsilon}\|w\|_{L^{3/2}([1,2],dr)}.
\end{align*} 
Likewise, if \eqref{dualWolff} holds, then by linearizing the maximal function, given $f\in L^{3}(\R^2,dz)$, for an appropriate $a(r)$ we have
\[
M_\delta f(r) = \frac{1}{|C_{\delta,a(r),r}|}\int_{C_{\delta,a(r),r}}|f(z)|\,dz.
\]
By duality, for an appropriate $w\in L^{3/2}([1,2],dr)$ with $\|w\|_{3/2} = 1$,
\begin{align*}
\|M_\delta f\|_{L^{3}([1,2],dr)} &= \int_1^2 M_\delta f(r) w(r)\,dr \\
&= \int_1^2\bigg(\frac{1}{|C_{\delta,a(r),r}|}\int_{C_{\delta,a(r),r}}|f(z)|\,dz\bigg)w(r)\,dr \\
&= \int_{\R^2}|f(z)|\bigg(\int_1^2 w(r)\frac{C_{\delta,a(r),r}}{|C_{\delta,a(r),r}|}(z)\,dr\bigg)\,dz \\
&\le \|f\|_{L^3(\R^2,dz)}\|m[w,a]\|_{L^{3/2}(\R^2,dz)}\\
&\le A_\epsilon \delta^{-\epsilon}\|f\|_{L^3(\R^2,dz)}.
\end{align*}
\end{proof}

The dual form of Wolff's circular maximal function estimate implies the following concrete formulation in terms of the multiplicity of a collection $\delta$-annuli.

\begin{example}\label{ex:wolff}
    Let $\{r_i\}_{i=1}^N\subset[1,2]$ be a $10\delta$-separated set of radii (not necessarily maximal). For real numbers $b_i$, set 
    \[
    w(r) = \sum_i b_i1_{[r_i,r_i+\delta]}(r),
    \]
    and for each $i$, let $a_i\in\R^2$ be an arbitrary point. Set
    \[
    a(r) = \sum_i a_i1_{[r_i,r_i+\delta]}(r).
    \]
    By definition,
    \[
    m[w,a](z)\approx \sum_{i}b_iC_{\delta,a_i,r_i}(z),
    \]
    where we remind of the notation $C(z) = 1_{C}(z)$. Hence Proposition \ref{prop:dual-wolff}, the dual form of Wolff's circular maximal estimate, implies
    \[
    \int_{\R^2} (\sum_{i=1}^Nb_iC_{\delta,a_i,r_i}(z))^{3/2}\,dz \lesssim \delta^{-\epsilon}\delta\sum_{i=1}^N|b_i|^{3/2}.
    \]
    As a special case, setting each $b_i = 1$ we get
    \[
    \int_{\R^2} (\sum_{i=1}^NC_{\delta,a_i,r_i}(z))^{3/2}\,dz \lesssim \delta^{-\epsilon}\delta N.
    \]
\end{example}

Dual circular maximal function estimates like Wolff's original estimate in the form of Example \ref{ex:wolff} or Pramanik--Yang--Zahl's generalization in Theorem \ref{thm:pramanik-yang-zahl} can be translated into estimates for the following $\delta$-discretized tangency counting problem.
\begin{prob}[Circle tangencies at scale $\delta$]\label{prob:pairs}
    If $X\subset Q$ is a set of $N$ points and $\delta>0$, estimate the number of ``circle tangencies'' at scale $\delta>0$:
    \[
    CT_\delta(X) = \#\{(x,x')\in X^2: ||\bar x-\bar x'| - |x_3 - x_3'|| < \delta\}.
    \]
    Here we write a point $x\in\R^3$ in coordinates as $x = (\bar x,x_3)\in \R^2\times \R$.
\end{prob}
After discussing preliminaries on circle tangencies, tangency rectangles, and point-circle duality, we will sketch how dual circular maximal function estimates can be used to address Problem \ref{prob:pairs} in Section \ref{sec:geometry-rect-lightplanks}. See Theorem \ref{thm:num-ll-pairs} for the precise version of the circle tangency estimate we will use as a lemma in the proof of Theorem \ref{thm:ortiz-sparse-l1}.

\subsection{Duality constructions}\label{sec:rect-lp} 
 We consider circles in the plane parametrized by their center-radius pairs. Given a point $(a,r)\in\R^2\times(0,\infty)$, we will equivalently regard it as the circle
\[
C_{a,r} = \{z\in\R^2:||z-a|-r|=0\}.
\]
See Figure \ref{fig:point-circ}. Likewise, the $\delta$-neighborhood (in $\R^3$) of the point $(a,r)\in\R^2\times\R$ is identified with the $\delta$-thin annulus
\[
C_{\delta,a,r} = \{z\in\R^2:||z-a|-r|<\delta\}.
\]
\begin{figure}
    \centering
    \includegraphics[width=7cm]{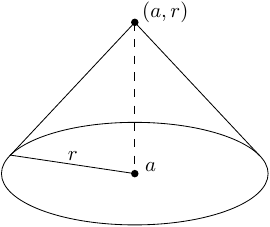}
    \caption{Point-circle duality}
    \label{fig:point-circ}
\end{figure}
Let $\epsilon>0$. We will assume that $\delta < \delta_0(\epsilon)$ is small enough so that $\delta_0^{\epsilon}< 10^{-3}$ to ensure that approximations such as $\cos\theta \sim 1-\theta^2/2$ hold up to constant factors if $|\theta|\le \delta^{\epsilon}$.
We identify points $x = (a,r)$ with their corresponding circles $C_{a,r}$. All the circles we consider will be parametrized by points in $Q=B(e_3,\alpha_0)$ for a small but absolute constant $\alpha_0>0$ unless mentioned otherwise. 

\begin{remark}[Notation]\label{rem:notation}
    We will make use of the following notation. Throughout,  $\epsilon > 0$ is fixed.
    \begin{itemize}
    \item $A\ll B$:  there is a constant $c>0$ so $A\lesssim \delta^{c\sqrt\epsilon}B$.
    \item $A\lessapprox B$:  there is a constant $C>0$ so that $A\lesssim \delta^{-C\epsilon} B$. Note that with this definition, as long as $\delta$ is sufficiently small depending on $\epsilon$, $A\ll B\lessapprox C$ implies $A\ll C$.
    \item $A\approx B$: $A\lessapprox B$ and $B\lessapprox A$ (with possibly different implied constants).

    \item $e_1,e_2,e_3$ will denote the standard basis vectors of $\R^3$.
    \item $\Gamma_0 = \{(a,r)\in\R^2\times\R:||a|-|r||=0\}$ is the lightcone with vertex $0$.
    \item $\Gamma_y = \Gamma_0 + y$ is the lightcone with vertex $y$.
    \item $Q=B(e_3,\alpha_0)$, the ball of radius $\alpha_0$ about $e_3$, which we take as the parameter space of circles.
    \item When we are using point-circle duality, we will not necessarily distinguish between circles and the points of $Q$ that parametrize them. For instance, we may say things like ``given a set of circles in $Q$,''  where we mean precisely consider a subset of $Q$ which we think of as a collection of circles.
    \item If $E$ is a set, we will use $E(x)$ to denote $1_E(x)$, the indicator function of $E$. 
    \item If $x=(a,r),x'=(a',r')\in Q$, then $\Delta(x,x') = ||a-a'|-|r-r'||$ is (up to an absolute constant) the distance from $x$ to $\Gamma_{x'}$, and vice-versa.
    \item If $(E_1,o_1),(E_2,o_2)$ are sets with designated ``centers'' $o_1\in E_1$, $o_2\in E_2$, and $A^{-1}E_1\subset E_2\subset AE_1$ for some $A\approx 1$, where the notation $AE_1$ denotes the dilation of $E_1$ by $A$ about its center $o_1$, then we may write $E_1\asymp E_2$. We will say that such sets are ($A$-)\emph{comparable}.
    \end{itemize}
\end{remark}

Given $X\subset Q$ a set of circles, define the multiplicity functions
\[
m_{\lambda\delta}(z) = \sum_{x\in X}C_{\lambda\delta,x}(z), \quad z\in\R^2,\quad \delta,\lambda>0.
\]
The number $m_{\lambda\delta}(z)$ counts the number of thin annuli the point $z$ belongs to. We want to understand the shape of the large level sets of $m_{\lambda\delta}$ for some small fixed $\delta$, and $\lambda\lessapprox 1$. We can partition these level sets (which are contained in the union of thin annuli) into curvilinear rectangles of width $\lambda\delta$, and variable length $0<\tau\lesssim 1$. It turns out that the range $\sqrt\delta\le \tau\ll 1$ is the most important for us to understand.

\begin{defn}[$\delta,\tau$-rectangle, core circle, center]
For $\delta^{1/2}\le \tau \ll 1$, a \emph{$\delta,\tau$-rectangle} $\Omega$ is the $\delta$-neighborhood of an arc of length $\tau$ on some circle of radius $r\in[1-\alpha_0,1+\alpha_0]$. We will sometimes refer to the implicit circle in this definition as the \emph{core circle} of $\Omega$, and we may write $\Omega = \Omega^{(v)}$ if $v$ is the core circle of $\Omega$. The midpoint of the core arc of $\Omega$ will be referred to as the \emph{center} of $\Omega$.
\end{defn}

\begin{figure}
    \centering
    \includegraphics[width=6cm]{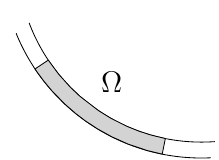}
    \caption{A $\delta,\tau$-rectangle $\Omega$, pictured in gray}
    \label{fig:delta-tau-rect}
\end{figure}

\begin{defn}[Comparable]\label{defn:rectangle-comparable}
We say two $\delta,\tau$-rectangles $\Omega_1,\Omega_2$ are \emph{$A$-comparable} if each is contained in the $A$-dilation of the other about their centers. If $\Omega_1,\Omega_2$ are $A$-comparable for some $A\approx 1$, then we simply say they are \emph{comparable}. If $\Omega_1,\Omega_2$ are not $A$-comparable, we say they are \emph{$A$-incomparable.} A collection $\mathcal R$ of $\delta,\tau$-rectangles is \emph{pairwise $A$-incomparable} if no two members of $\mathcal R$ are $A$-comparable.
\end{defn}
If $\tau\lesssim \delta^{1/2}$, then a $\delta,\tau$-rectangle is approximately a rectangle in the usual sense, while if $\tau$ is much larger than $\delta^{1/2}$, a $\delta,\tau$-rectangle is a curvilinear rectangle.

\begin{defn}[Tangency]\label{defn:tangency}
We say a $\delta,\tau$-rectangle $\Omega$ is \emph{$\lambda$-tangent} to the circle $x$ if $\Omega\subset C_{\lambda\delta,x}$. We let $\mathbf D_{\lambda\delta}(\Omega) = \{x\in Q:\Omega\subset C_{\lambda\delta,x}\}$ be the collection of $\lambda$-tangent circles to $\Omega$ in $Q$.
\end{defn}
Note that in the definition of tangency, we restrict the ``dual set'' $\mathbf D_{\lambda\delta}(\Omega)$ to be contained in $Q = B(e_3,\alpha_0)$. The terminology ``dual set'' will be appropriately justified in Theorem \ref{thm:rect-lp-duality}.

We record the following useful and easily proved facts about taking tangency.
\begin{prop}[Properties of tangency]\label{prop:simple-duality}
For every $\delta>0$, the following hold.
\begin{enumerate}
    \item[(i)] (Monotonicity) If $\Omega'\subset \Omega$, then $\mathbf D_\delta(\Omega)\subset\mathbf D_\delta(\Omega')$.

    \item[(ii)] (Intersection) If $\Omega = \bigcup_k\Omega_k$, then $\mathbf D_\delta(\Omega) = \bigcap_k\mathbf D_\delta(\Omega_k)$
    \item[(iii)] For every $\delta,\tau$-rectangle $\Omega = \Omega^{(v)}$, we have $v\in \mathbf D_\delta(\Omega^{(v)})$ and $B(v,\delta)\cap Q\subset\mathbf D_{10\delta}(\Omega)$. In particular, $\mathbf D_\delta(\Omega^{(v)})\ne\emptyset$ for every $v\in Q$.
    \end{enumerate}
\end{prop}
We will refine property (iii) of Proposition \ref{prop:simple-duality} substantially in Proposition \ref{prop:dual-to-d,t-rectangle}. If $\Omega=\Omega^{(v)}$ is a $\delta,\tau$-rectangle, then the core circle $v$ is $1$-tangent to $\Omega$. Besides $v$, there are other nearby circles $w\in Q$ which are $\approx 1$-tangent to $\Omega$. The set of all such $w$ takes the shape of an ``essentially unique'' $\approx \delta\times\delta\tau^{-1}\times\delta\tau^{-2}$-\emph{lightplank} centered on $v$ when regarded as a subset of $\R^3$. To describe the shape of $\mathbf D_{10\delta}(\Omega)$ more precisely, we introduce the following definition.

\begin{defn}\label{defn:lightlike-basis}
A \emph{lightlike basis} for $\R^3$ is an orthonormal basis $\mathcal E = e_m,e_l,e_s$ of $\mathbb R^3$ such that for some $\theta\in[-\pi/2,\pi/2)$, with respect to the standard basis of $\R^3$,
\[
e_m = \begin{pmatrix}-\sin\theta \\ \cos\theta \\ 0\end{pmatrix},e_l = \frac{1}{\sqrt 2}\begin{pmatrix}-\cos\theta \\ -\sin\theta \\ 1\end{pmatrix},e_s = \frac1{\sqrt 2}\begin{pmatrix}\cos\theta \\ \sin\theta \\ 1\end{pmatrix}.
\]
A \emph{lightlike coordinate system} $(x_m,x_l,x_s),o$ is the Cartesian coordinate system with respect to a lightlike basis with the point $o\in\R^3$ as the designated origin, i.e., for any $x$ in $\R^3$, $x = o + x_me_m + x_le_l + x_se_s$. If we want to emphasize the angle $\theta$, we will write $x_m(\theta),x_l(\theta),x_s(\theta),e_m(\theta),e_l(\theta),e_s(\theta)$.
\end{defn}
Note that with this definition, a lightlike basis is ordered and right-handed, and in particular, it is completely determined by the first basis vector.

\begin{defn}[Lightplank, comparable, essentially unique]\label{def:lightplank}
    Let $\delta^{1/2}\le \tau\ll 1$. A rectangular parallelepiped $P\subset\R^3$ is a $\approx \delta\times\delta\tau^{-1}\times\delta\tau^{-2}$-$\emph{lightplank}$ if the edge lengths of $P$ are $A_1\delta\le A_2\delta\tau^{-1}\le A_3\delta\tau^{-2}$ for numbers $A_1,A_2,A_3\approx 1$, and if for some lightlike basis $e_m,e_l,e_s$ of $\R^3$, the edges of $P$ are parallel to the vectors $e_m,e_l,e_s$, in the order ``intermediate, long, short.'' We will adopt a similar notation as with $\delta,\tau$-rectangles where we write $P = P^{(v)}$ if $v\in \R^3$ is the center of $P$.

    Two lightplanks $P,P'$ are \emph{comparable} if they are both contained in the $A$-dilation of the other for some $A\approx 1$.

    We say a lightplank $P$ satisfying a property $\pi$ is \emph{essentially unique} if any other lightplank $P'$ which also satisfies $\pi$ is comparable to $P$.
\end{defn}

The following calculation of $\mathbf D_{10\delta}(\Omega)$ when $\Omega$ is a $\delta,\sqrt\delta$-rectangle will be the basis for the calculation of $\mathbf D_{\lambda\delta}(\Omega)$ when $\Omega$ is a $\delta,\tau$-rectangle, and $\lambda\approx 1$, $\tau \gg \sqrt\delta$.
\begin{prop}[Dual of $\delta,\sqrt\delta$-rectangle]\label{prop:root-delta}
Let $o = e_3$ and $\Omega^{(o)} = [1-\delta,1+\delta]\times[-\sqrt\delta/2,\sqrt\delta/2]$. Let
\[
e_m = e_2, \quad e_l = \frac{-e_1 + e_3}{\sqrt 2}, \quad e_s = \frac{e_1 + e_3}{\sqrt 2}
\]
be a lightlike basis for $\R^3$, and let $(x_m,x_l,x_s),o$ be the associated lightlike coordinate system. Recall that $Q = B(e_3,\alpha_0)$. If $C$ is a sufficiently large absolute constant, then the following hold. 
\begin{enumerate}
    \item[(i)] If $x\in P^{(o)}\cap Q$, where
    \[
    P^{(o)} = \{(x_m,x_l,x_s):|x_m|\le \sqrt\delta, |x_l|\le 1,|x_s|\le \delta\}
    \]
    then $x\in C\cdot \mathbf D_{10\delta}(\Omega^{(o)})$, the dilation of $\mathbf D_{10\delta}(\Omega^{(o)})$ about $o$.
    \item[(ii)] If $x\in \mathbf D_{10\delta}(\Omega^{(o)})$, then $x \in Q\cap CP^{(o)}$, where $CP$ is the dilation of $P$ by a factor of $C$ about its center $o$.
\end{enumerate}
\end{prop}
\begin{remark}
    Proposition \ref{prop:root-delta} shows in what precise sense the $\delta,\sqrt\delta$-rectangle $\Omega^{(o)}$ is dual to ``the'' $\delta\times\sqrt\delta\times1$-lightplank $P^{(o)}$: any other lightplank $P'$ which satisfies (i) and (ii)  is $O(1)$-comparable to $P^{(o)}$. Another way to say it is that there is an essentially unique $\approx\delta\times\sqrt\delta\times 1$-lightplank satisfying (i) and (ii) in the sense of Definition \ref{def:lightplank}. 
\end{remark}
\begin{proof}[Proof of Proposition \ref{prop:root-delta}]
    First we prove (i). The reader can verify that
    \[
    \big(\mathbf D_{10\delta}(\Omega^{(o)}),o\big)\asymp \big(\mathbf D_{C\delta}(\Omega^{(o)}),o\big),\quad C>1,
    \]
    so it suffices to show $P^{(o)}\subset \mathbf D_{C\delta}(\Omega^{(o)})$ for absolute $C$.
    
    By definition of $P^{(o)}$ and the lightlike coordinate system $(x_m,x_l,x_s),o$, it is clear that $P^{(o)}\subset Q$. Let $x\in P^{(o)}$, and let $(x_1,x_2,x_3)$ be the usual Euclidean coordinates of $x$ with respect to the standard basis $e_1,e_2,e_3$. To verify $x\in\mathbf D_{C\delta}(\Omega^{(o)})$, let $(a_1,a_2) \in\Omega^{(o)}$ be arbitrary, and consider the quantity
    \begin{equation}\label{eqn:lightplank-dual-i}
    I = |(x_1-a_1)^2 + (x_2-a_2)^2 - x_3^2|.
    \end{equation}
    We need to show that $I\le C\delta$ for a constant that does not depend on $(a_1,a_2)$. Let $\bar x = (x_1,x_2)$, and $a_1 = 1 + h$ with $|h|\le \delta$. Expanding Equation \eqref{eqn:lightplank-dual-i} and applying the triangle inequality,
    \begin{align*}
    I &= |\big((x_1 - 1)^2 + x_2^2 - x_3^2\big) + h^2 -2(x_1-1)h - 2x_2a_2 + a_2^2|\\
    &\le \big(||\bar x-e_1|-x_3|\cdot||\bar x-e_1|+x_3|\big)+\delta^2 + 2\delta|x_1-1| + 2\sqrt\delta|x_2| + \delta.
    \end{align*}
    Now we use that $x\in P^{(o)}$ to make some estimates of this last expression. First, $||\bar x-e_1|-x_3| = \Delta(x,e_1) \le C\delta$, because $\Delta(x,e_1)$ is equal to the distance from $x$ to the lightcone with vertex $e_1$ within constant multiplicative factors. The quantity $||\bar x-e_1|+x_3| < 10$ because $x \in Q$. Similarly, $|x_1-1|\le 2$. Note $|x_2| = |x_m| \le \sqrt\delta$, so $2\sqrt\delta|x_2|\le 2\delta$. Altogether this shows that $I\le C\delta$, as desired.

    Now we prove (ii). Let $\ell = \{e_1 + t(e_3-e_1):t\in\R\}$ be a lightray intersecting $e_1$. Consider the infinite $2\delta\times\sqrt\delta$ rectangular prism $R$ we get by sliding $\Omega^{(o)}$ along $\ell$. (Concretely, $R$ is the Minkowski sum $e_1 + (\ell-e_1) + (\Omega^{(o)}-e_1)$.) If $x\in Q$, then for large enough $C>1$ (independent of $\delta$), the dilation $CR$ will contain $P^{(o)}$. This shows that if $x\in \mathbf D_{10\delta}(\Omega^{(o)})$, then $x\in Q\cap CP^{(o)}$ for an appropriate $C>1$.
\end{proof}

\begin{remark}[Our favorite position, coordinates]\label{rem:favorite-posn}
    The coordinate system and position of $\Omega^{(o)}$ in Proposition \ref{prop:root-delta} are so convenient for computations that we will say a $\delta,\tau$-rectangle $\Omega^{(o)}$ is \emph{in our favorite position} if $o = e_3$, and the center of $\Omega^{(o)}$ as defined in Definition \ref{def:lightplank} is $e_1$.

    If $P^{(o)}$ is a $\delta\times\delta\tau^{-1}\times\delta\tau^{-2}$-lightplank, we say $P^{(o)}$ is \emph{in our favorite position} if $o = e_3$ is the center of $P^{(o)}$, and if the intermediate axis of $P^{(o)}$ is parallel to $e_m:=e_2$, the long axis is parallel to $e_l:=\frac{-e_1+e_3}{\sqrt 2}$, and the short axis is parallel to $e_s:=\frac{e_1+e_3}{\sqrt 2}$.

    We say $(x_m,x_l,x_s),o$ is \emph{our favorite lightlike coordinate system.}
\end{remark}

By changing coordinates, we can study any $\delta,\sqrt\delta$-rectangle by first changing coordinates so that the transformed rectangle is in our favorite position, applying Proposition \ref{prop:root-delta}, and then transforming back to the original coordinates. Alternatively, we have the following ``coordinate-invariant'' description of an essentially unique dual lightplank to a $\delta,\sqrt\delta$-rectangle in the plane.

\begin{construction}[Dual of $\delta,\sqrt\delta$-rectangle]\label{const:dual-root-delta}
    Proposition \ref{prop:root-delta} describes how to construct an essentially unique $\delta\times\sqrt\delta\times 1$-lightplank dual to a $\delta,\sqrt\delta$-rectangle in our favorite position. Now we consider more arbitrary $\delta,\sqrt\delta$-rectangles.

    Let $\Omega^{(v)}$ be an arbitrary $\delta,\sqrt\delta$-rectangle in the plane with core circle $v\in Q = B(e_3,\alpha_0)$ and center $c_\Omega$. Let $e_m$ be a unit vector in $\R^2$ parallel to the long edge of $\Omega^{(v)}$, forming an angle in $[-\pi/2,\pi/2)$ with the standard basis vector $e_1$. Consider the lightlike basis $e_m,e_l,e_s$ determined by $e_m$ and the lightrays $\ell_+=c_\Omega + \R e_l,\ell_-=c_\Omega+\R e_s$ passing through $c_\Omega$.
    
    Define infinite $\delta\times\sqrt\delta$ rectangular prisms $R_+,R_-$ by
    \[
    R_+ = \Omega^{(v)}-c_{\Omega} + \ell_+,\quad R_- = \Omega^{(v)}-c_\Omega + \ell_-.
    \]
    Since $v\in Q$, $v$ lies in exactly one of the sets
    \[
    Q\cap R_+, \quad Q\cap R_-.
    \]
    By relabeling $R_+,R_-$ if necessary, we may assume $v\in Q\cap R_+$. Then $P^{(v)}=R_+\cap 2Q$ is an essentially unique $\approx \delta\times\sqrt\delta\times 1$-lightplank $P^{(v)}$ satisfying $P^{(v)}\asymp \mathbf D_{10\delta}(\Omega^{(v)})$.
\end{construction}

It will be convenient for future computations to describe how different lightlike coordinates are related when $o$ is the same.
\begin{prop}[Change of coordinates, fixed $o$]\label{prop:change-of-coordinates}
    If $\theta\in[-\pi/2,\pi/2)$ and $o\in Q$ is fixed, the following relationship between the lightlike coordinates $(x_m(\theta),x_l(\theta),x_s(\theta)),o$ and $(x_m(0),x_l(0),x_s(0)),o$ holds:
    \begin{equation}\label{eqn:change-coord}
        \begin{pmatrix}
            \cos \theta & \dfrac{-\sin\theta}{\sqrt 2} & \dfrac{\sin\theta}{\sqrt 2} \\
            \dfrac{\sin\theta}{\sqrt2} & \dfrac{1+\cos\theta}{2} & \dfrac{1-\cos\theta}{2} \\
            \dfrac{-\sin\theta}{\sqrt 2} & \dfrac{1-\cos\theta}{2} & \dfrac{1+\cos\theta}{2}
        \end{pmatrix}
        \begin{pmatrix}
            x_m(\theta) \\ x_l(\theta) \\ x_s(\theta)
        \end{pmatrix}
        =
        \begin{pmatrix}
            x_m(0) \\ x_l(0) \\ x_s(0)
        \end{pmatrix}.
    \end{equation}

\end{prop}
\begin{proof}
    Since the lightlike bases $e_m(\theta),e_l(\theta),e_s(\theta)$ and $e_m(0),e_l(0),e_s(0)$ are orthonormal, the proof is just the calculation of the nine inner products $\langle e_m(\theta),e_m(0)\rangle$, $\langle e_l(\theta),e_m(0)\rangle,\dots$, etc.
\end{proof}

The next Proposition is the first of our dictionary results which does not appear in Wolff's work \cite{wolff2000local}.

\begin{prop}[Dual of $\delta,\tau$-rectangle]\label{prop:dual-to-d,t-rectangle}
Let $\delta^{1/2}<\tau\ll 1$, and let $\Omega^{(o)}$ be a $\delta,\tau$-rectangle with core circle $o=e_3$ in our favorite position. In our favorite lightlike coordinate system $(x_m,x_l,x_s),o$, let
\[
P^{(o)} = \{(x_m,x_l,x_s):|x_m|\le \delta\tau^{-1},|x_l|\le \delta\tau^{-2},|x_s|\le\delta\}.
\]
Recall that $Q = B(e_3,\alpha_0)$. If $C$ is a sufficiently large absolute constant, then the following hold.
\begin{enumerate}
    \item[(i)] If $x\in P^{(o)}\cap Q$, then $x\in C\cdot\mathbf D_{10\delta}(\Omega^{(o)})$.
    \item[(ii)] If $x\in \mathbf D_{10\delta}(\Omega^{(o)})$, then $x\in Q\cap CP^{(o)}$, where $CP^{(o)}$ is the dilation of $P^{(o)}$ by a factor of $C$ about its center, $o$.
\end{enumerate}
\end{prop}
\begin{proof}[Proof snapshot]
See Figure \ref{fig:dual-lightplank}. Assume $\Omega^{(o)}$ is in our favorite position and take the intersection of lightplanks comparable to $\mathbf D_{10\delta}(\Omega_k^{(o)})$ from Construction \ref{const:dual-root-delta} for some $\delta,\sqrt\delta$-rectangles $\{\Omega_k^{(o)}\}_k$ which cover $\Omega^{(o)}$ but not $2\Omega^{(o)}$.
\end{proof}
\begin{figure}
    \centering
    \includegraphics[width=10cm]{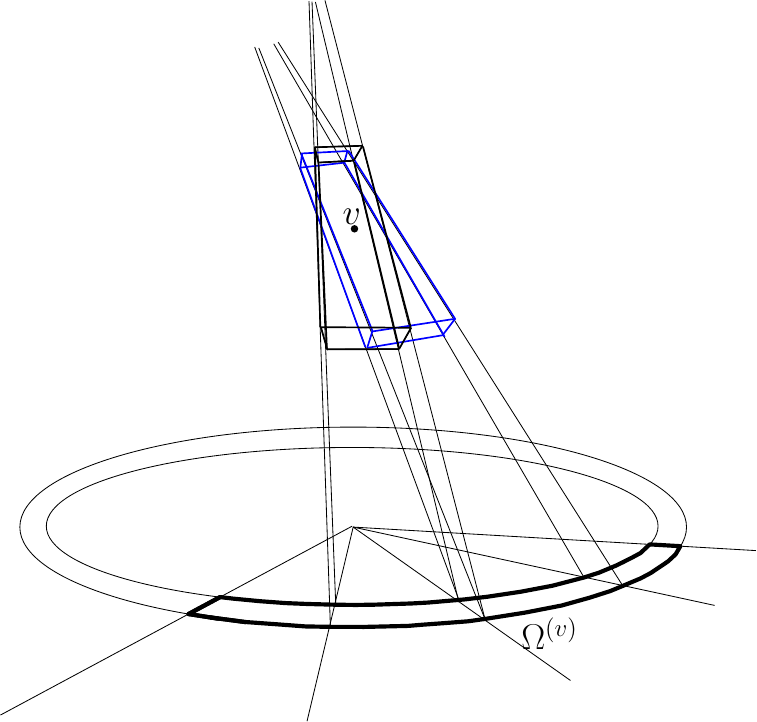}
    \caption{A snapshot of ``the'' dual $\delta\times\delta\tau^{-1}\times\delta\tau^{-2}$-lightplank to a $\delta,\tau$-rectangle $\Omega^{(v)}$}
    \label{fig:dual-lightplank}
\end{figure}

\begin{proof}
    First we prove (i). We begin by making a reduction based on taking tangency in Proposition \ref{prop:simple-duality}. Let $\tau<N\sqrt\delta<2\tau$, and for each $k\in \mathbb Z$, $|k|\le N$, consider the $\delta,\sqrt\delta$-rectangle
    \[
    \Omega_k^{(o)} = \{z\in\R^2:||z|-1|<\delta,|z-e^{ik\sqrt\delta}|<\frac12\sqrt\delta\}, \quad k\in\mathbb Z, |k|\le N.
    \]
    and let
    \[
    \overline\Omega^{(o)}_N = \bigcup_{|k|\le N}\Omega_k^{(o)}.
    \]
    It is clear from the definition of $N$ that $\Omega^{(o)}\subset\overline\Omega^{(o)}_N$ so by the monotonicity property of Proposition \ref{prop:simple-duality}, we have $\mathbf D_{10\delta}(\overline\Omega^{(o)}_N)\subset\mathbf D_{10\delta}(\Omega^{(o)})$. It thus suffices to show
    \begin{equation}\label{eqn:duality-delta-tau}
    P^{(o)}\subset C\cdot \mathbf D_{10\delta}(\overline\Omega^{(o)}_N)
    \end{equation}
    for absolutely large $C$.
    By the intersection property of Property \ref{prop:simple-duality}, it is sufficient to show
    \[
    P^{(o)}\subset C\cdot \bigcap_{k=1}^N \mathbf D_{10\delta}(\Omega_k^{(o)}).
    \]
    By Construction \ref{const:dual-root-delta} or Proposition \ref{prop:root-delta}, for each $|k|\le N$, let $P_k^{(o)}$ be an essentially unique $\delta\times\sqrt\delta\times 1$-lightplank $P_k^{(o)}$ such that
    \[
    P_k^{(o)}\subset C\cdot \mathbf D_{10\delta}(\Omega_k^{(o)})
    \]
    for absolutely large $C$ independent of $k$. It follows that
    \[
    \bigcap_{k=1}^N P^{(o)}_0\cap P^{(o)}_{k}\cap P^{(o)}_{-k} \subset C\cdot \bigcap_{|k|\le N}\mathbf D_{10\delta}(\Omega_k^{(o)})
    \]
    for absolutely large $C$. Finally, we have reduced (i) to showing
    \[
    P^{(o)}\subset C\cdot \bigcap_{k=1}^NP^{(o)}_0\cap P^{(o)}_{k}\cap P^{(o)}_{-k}.
    \]
    
     \textbf{Claim.} For each $1\le k\le N$, in our favorite lightlike coordinate system $(x_m,x_l,x_s),o$ (see Remark \ref{rem:favorite-posn}),
    \[
    P_0^{(o)}\cap P_k^{(o)} \cap P_{-k}^{(o)}\asymp \{|x_m|\le k^{-1}\sqrt\delta,|x_l|\le k^{-2},|x_s|\le \delta\},
    \] 
    with an absolute constant of comparability.

    Given the claim, the proof of (i) follows because $N\sqrt\delta \sim \tau$, so
    \begin{align*}
    P^{(o)} &= \{|x_m|\le \delta\tau^{-1},|x_l|\le \delta\tau^{-2},|x_s|\le\delta\}\\
    &\asymp\{|x_m| \le N^{-1}\sqrt\delta,|x_l|\le N^{-2},|x_s|\le \delta\} \\
    &\asymp P_0^{(o)}\cap P_N^{(o)}\cap P_{-N}^{(o)}\\
    &\asymp \bigcap_{k=1}^NP_0^{(o)}\cap P_k^{(o)} \cap P_{-k}^{(o)} \subset C\cdot  \bigcap_{|k|\le N} \mathbf D_{10\delta}(\Omega_k^{(o)}).
    \end{align*}
    The third $\asymp$ in the chain above holds because the sets $\{|x_m|\le k^{-1}\sqrt\delta,|x_l|\le k^{-2},|x_s|\le\delta\}\asymp P_0^{(o)}\cap P_k^{(o)}\cap P_{-k}^{(o)}$ are decreasing as $k$ increases.

    To prove the claim, let $\theta\in \{\pm k\sqrt\delta\}$ and consider the lightlike coordinate system $(x_m(\theta),x_l(\theta),x_s(\theta)),o$. The lightplank $P_{\frac{\theta}{\sqrt\delta}}^{(o)}$ is contained in the $\delta$-neighborhood of the affine plane
    \[
    \Pi_{\frac{\theta}{\sqrt\delta}}=o+\{x\in\R^3:\langle x,e_s(\theta)\rangle = 0\}.
    \]
    To see how $P_{\frac{\theta}{\sqrt\delta}}^{(o)}$ intersects $P_0^{(o)}$, we just need to determine the $(x_m,x_l,x_s),o$-coordinates of the four ``corner points'' $x\in \Pi_{\frac{\theta}{\sqrt\delta}}\cap \partial P_0^{(o)}$ (see Figure \ref{fig:lightplank-plane-points}).
    \begin{figure}
        \centering
        \includegraphics[width=13cm]{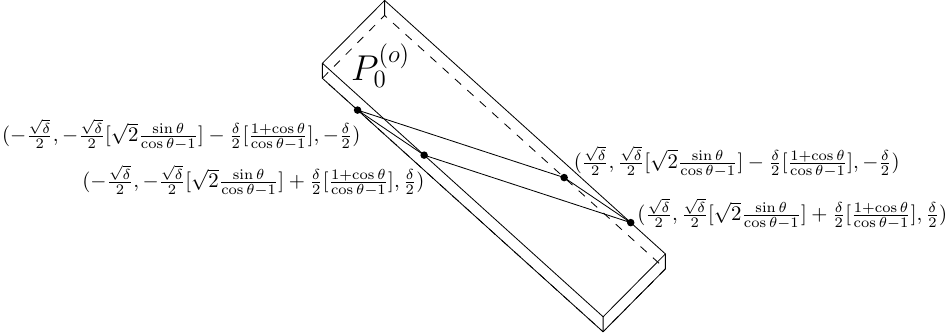}
        \caption{The four ``corner points'' of $P^{(o)}_0\cap P^{(o)}_{\frac{\theta}{\sqrt\delta}}$}
        \label{fig:lightplank-plane-points}
    \end{figure}
    The calculation to determine the coordinates shown in Figure \ref{fig:lightplank-plane-points} is not difficult, but it is lengthy, so we just describe the method we use to obtain the coordinates here. The four marked points in Figure \ref{fig:lightplank-plane-points} are the ``corner points'' where the plane $\Pi_{\frac{\theta}{\sqrt\delta}}$ intersects $P_0^{(o)}$ on its faces $x_m = \pm\frac{\sqrt\delta}{2}$ and $x_s = \pm\frac{\delta}{2}$. We use Equation \eqref{eqn:change-coord} to solve for $x_m(\theta),x_l(\theta)$ (we know $x_s(\theta) = 0$ by definition of the plane $\Pi_{\frac{\theta}{\sqrt\delta}}$) in terms of the known values for $x_m,x_s$, and then use Equation \eqref{eqn:change-coord} again with $x_m=\pm\frac{\sqrt\delta}{2},x_s=\pm\frac{\delta}{2}$ to determine the value of $x_l$ for each of these four points.

    By computing these points for each $\theta\in\{\pm k\sqrt\delta\}$, we can find the region of overlap $P_0^{(o)}\cap P_k^{(o)}\cap P_{-k}^{(o)}$ in the $(x_m,x_l,x_s),o$-coordinates---see Figure \ref{fig:inner-lightplank}.

    The upshot is that for each $1\le k\le N$ we have
    \begin{align*}
    P_0^{(o)}\cap P_k^{(o)} \cap P_{-k}^{(o)}&\asymp \{|x_m|\le k^{-1}\sqrt\delta,|x_l|\le k^{-2},|x_s|\le \delta\},
    \end{align*}
    which finishes the proof of the claim, and the proof of (i) that $P^{(o)}\cap Q\subset C\cdot  \mathbf D_{10\delta}(\Omega^{(o)})$ provided $C$ is sufficiently large.
    \begin{figure}
        \centering
        \includegraphics[width=7cm]{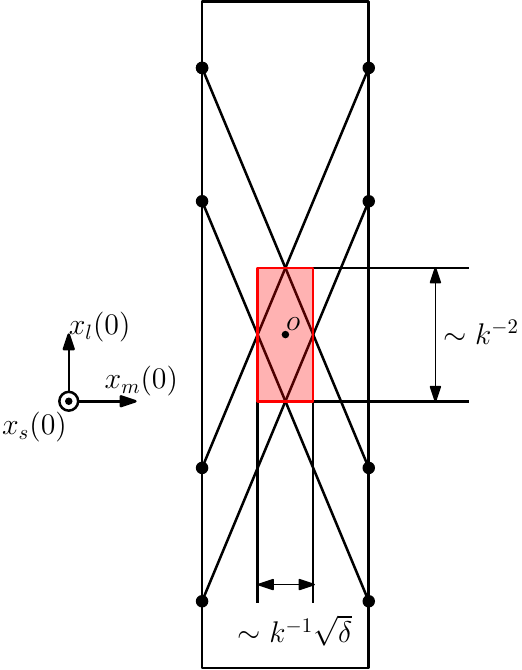}
        \caption{The dimensions of the intersection $P_0^{(o)}\cap P_k^{(o)}\cap P_{-k}^{(o)}$}
        \label{fig:inner-lightplank}
    \end{figure}

    The proof of (ii) now follows partly by the work we did to prove (i), so we avoid repeating some of the same details. We choose $N$ so that $\tau/2<N\sqrt\delta<\tau$, and let $\Omega_k^{(o)}$ for $|k|\le N$ be $\delta,\sqrt\delta$-rectangles as in the proof of (i). Let $\underline\Omega^{(o)}_N = \bigcup_{|k|\le N}\Omega^{(o)}_k$ so $\mathbf D_{10\delta}(\Omega^{(o)})\subset \mathbf D_{10\delta}(\underline\Omega^{(o)}_N)$ by monotonicity. By the intersection property of taking tangency, we have
    \begin{equation}\label{eqn:x-contained-duality}
        \mathbf D_{10\delta}(\Omega^{(o)})\subset \bigcap_{|k|\le N}\mathbf D_{10\delta}(\Omega_k^{(o)}),
    \end{equation}
    so it suffices to show the latter set is contained in $CP^{(o)}$.
    By the calculation we established in the claim from the proof of (i), in our favorite coordinate system, the intersection on the right-hand side of Equation \eqref{eqn:x-contained-duality} is comparable to the $\delta\times N^{-1}\sqrt\delta\times N^{-2}$-lightplank
    \[
    \{|x_m|\le N^{-1}\sqrt\delta,|x_l|\le N^{-2},|x_s|\le\delta\},
    \]
    which is itself $O(1)$-comparable to $P^{(o)}$ by the definition $N\sqrt\delta\sim \tau$. This finishes the proof of (ii), and the proposition.
\end{proof}
We can compute $\mathbf D_{10\delta}(\Omega^{(v)})$ for any particular $\delta,\tau$-rectangle $\Omega^{(v)}$ by translating our coordinates so that the transformed rectangle is in our favorite position, applying Proposition \ref{prop:dual-to-d,t-rectangle}, and then undoing the coordinate transformation. Alternatively, we record the following ``coordinate-invariant'' description of the dual lightplank to a $\delta,\tau$-rectangle analogous to Construction \ref{const:dual-root-delta}.
\begin{construction}[Dual to $\delta,\tau$-rectangle, $\sqrt\delta\le\tau\ll 1$]\label{const:dual-to-delta-t-rectangle}
    Given an arbitrary $\delta,\tau$-rectangle $\Omega^{(v)}$ with $v\in Q$, take the intersection of essentially unique $\delta\times\sqrt\delta\times 1$-lightplanks dual to sub-$\delta,\sqrt\delta$-rectangles contained in $\Omega^{(v)}$. The result is an essentially unique $\delta\times\delta\tau^{-1}\times\delta\tau^{-2}$-lightplank with long edge parallel to the lightray connecting $v$ and the center of $\Omega^{(v)}$.
\end{construction}

\begin{remark}\label{rem:continuous-construction}
    Construction  \ref{const:dual-to-delta-t-rectangle} is ``continuous'' in the sense that if $\Omega^{(v)}$ and $\Omega^{(w)}$ are comparable $\delta,\tau$-rectangles for some $v,w\in Q$, then $\mathbf D_{10\delta}(\Omega^{(v)}),\mathbf D_{10\delta}(\Omega^{(w)})$ are comparable $\approx\delta\times\delta\tau^{-1}\times\delta\tau^{-2}$-lightplanks. A precise version of this remark appears later in Proposition \ref{prop:change-of-basis}.
\end{remark}

Like the sets $\mathbf D_\delta(\Omega)\subset Q$ for $\Omega\subset\R^2$, there is an appropriate ``dual'' for subsets $E\subset Q = B(e_3,\alpha_0)$.
\begin{defn}
If $E\subset Q$, and $\delta>0$, let $$\mathbf D^*_{\delta}(E) = \{z\in\R^2:E\subset \Gamma_{\delta,z}\}$$ where $\Gamma_{\delta,z}$ is the $\delta$-neighborhood of $\Gamma_z = z+\Gamma_0$, and we remind that $\Gamma_0 = \{(a,r)\in\R^2\times\R:||a|-|r||=0\}$ is the lightcone with vertex $0$.
\end{defn}
Like $\mathbf D_\delta$, the map $\mathbf D_\delta^*$ obeys a few simple properties. 
\begin{prop}\label{prop:simple-duality-2}
For every $\delta>0$, the following hold.
\begin{enumerate}
    \item[(i)] (Monotonicity) If $E'\subset E$, then $\mathbf D_\delta^*(E)\subset\mathbf D_\delta^*(E')$.

    \item[(ii)] (Intersection) If $E = \bigcup_kE_k$, then $\mathbf D_\delta^*(E) = \bigcap_k\mathbf D_\delta^*(E_k)$.
    \item[(iii)] For every $\sqrt\delta\le\tau\ll 1$, and every $\delta\times\delta\tau^{-1}\times\delta\tau^{-2}$-rectangle $P$, $\mathbf D_{10\delta}^*(P)\ne\emptyset$.
    
    If $P = P^{(o)}$ is a $\delta\times\delta\tau^{-1}\times\delta\tau^{-2}$ lightplank in our favorite position, and $\Omega^{(o)}$ is a $\delta,\tau$-rectangle in our favorite position, then
    \[
    \mathbf D_{10\delta}^*(P^{(o)})\asymp \Omega^{(o)}.
    \]
    \item[(iv)] (Continuity) For $v,w\in Q$, if $P^{(v)}\asymp P^{(w)}$ then $\mathbf D^*_{10\delta}(P^{(v)})\asymp \mathbf D^*_{10\delta}(P^{(w)})$.
    \end{enumerate}
\end{prop}
\begin{proof}
    Points (i) and (ii) are immediate from the definition.

    Points (iii) and (iv) follow from Proposition \ref{prop:dual-to-d,t-rectangle}.
\end{proof}

To summarize the results of this Section, we record the following Theorem and accompanying Figure \ref{fig:rect-plank-duality}, which justifies our use of the terms ``dual lightplank'' and ``dual rectangle.''
\begin{theorem}[Rectangle-lightplank duality]\label{thm:rect-lp-duality}
    If $\Omega^{(v)}$ is a $\delta,\tau$-rectangle with $v\in Q$, then for appropriate $C\approx 1$,
    \[
    \mathbf D^*_{C\delta}(\mathbf D_{10\delta}(\Omega^{(v)})) \asymp \Omega^{(v)}.
    \]
    Likewise, if $v\in Q$ and $P^{(v)}$ is a $\delta\times\delta\tau^{-1}\times\delta\tau^{-2}$-lightplank, then
    \[
    \mathbf D_{C\delta}(\mathbf D^*_{10\delta}(P^{(v)}))\asymp P^{(v)}.
    \]
\end{theorem}

\begin{figure}
    \centering
    \includegraphics[width=12cm]{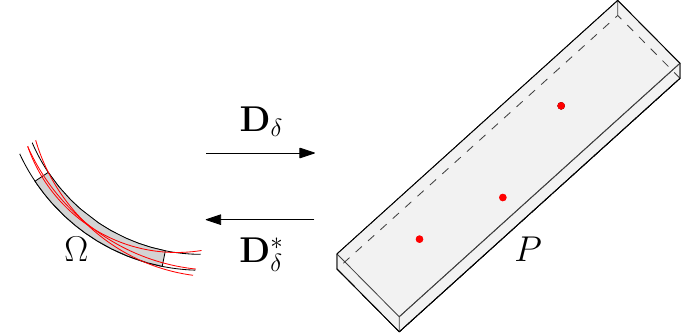}
    \caption{Rectangle-lightplank duality}
    \label{fig:rect-plank-duality}
\end{figure}

Until now, we have considered one rectangle $\Omega^{(v)}$ and a corresponding dual lightplank $P^{(v)}\asymp \mathbf D_{10\delta}(\Omega^{(v)})$. To address the $\delta$-discretized version of Problem \ref{prob:pairs}, we will also have to consider how multiple rectangles and corresponding lightplanks relate to one another.

\subsection{Tangency counting by dual circular maximal estimates}\label{sec:geometry-rect-lightplanks}
With the fundamentals of point-circle duality established, we will sketch how we count circle tangencies using dual circular maximal function estimates as a black box. We assume the family of circles $X$ we consider satisfies the 1-dimensional Frostman non-concentration condition
\begin{equation}\label{assump:circ-NC}
    |X\cap B_r|\lesssim \frac{r}{\delta},\quad\text{for all $r$-balls $B_r\subset\R^3$ and $r\ge \delta$}.
\end{equation}
Recall Pramanik, Yang, and Zahl's estimate from the introduction:
\begin{theorem}[Pramanik--Yang--Zahl \cite{pramanik2022furstenberg}, 2022]\label{thm:pramanik-yang-zahl}
    For each $\epsilon>0$, there is a constant $C_\epsilon$ so the following holds. Suppose $X\subset Q$ is a set of $\delta$-separated circles obeying the 1-dimensional Frostman non-concentration condition \eqref{assump:circ-NC}. Then the following estimate holds:
    \begin{equation}\label{eqn:pyz2}
    \int_{\R^2}\big(\sum_{x\in X}C_{\delta,x}(z)\big)^{3/2}\,dz \le C_\epsilon \delta^{-\epsilon}\delta|X|.
    \end{equation}
\end{theorem}
 
First we will sketch how we plan to use Theorem \ref{thm:pramanik-yang-zahl} to address Problem \ref{prob:pairs}. It will take some justification (see Proposition \ref{prop:main-geom}), but the integral appearing on the left-hand side of \eqref{eqn:pyz2} controls the cardinality of any maximal pairwise incomparable collection $\mathcal R$ of $\delta,\tau$-rectangles $\Omega$ contained in $\bigcup_{x\in X}C_{\delta,x}$ satisfying 
\begin{equation}\label{eqn:rich-rectangles}
    |X\cap \mathbf D_{10\delta}(\Omega)|\approx\mu,\quad\text{for all $\Omega\in \mathcal R$}
\end{equation}
to the effect of 
\begin{equation}\label{eqn:max-controls-rect}
    \int_{\bigcup\mathcal R}\big(\sum_{x\in X}C_{10\delta,x}(z)\big)^{3/2}\,dz\ge \mu^{3/2}\int_{\bigcup\mathcal R}\sum_{\Omega\in\mathcal R}\Omega(z)\,dz = \mu^{3/2}\delta\tau|\mathcal R|.
\end{equation}
If $v,w\in X$ is a pair of points that is $D$-separated, and nearly lightlike separated, meaning $\Delta(v,w)<\delta$, then $v,w$ belong to a common $\approx\delta\times \sqrt{D\delta}\times D$-lightplank $P$. Therefore, if
\[
\mathcal L_{D,<\delta}(X)=\{(v,w)\in X\times X:d(v,w)\sim D, \Delta(v,w)<\delta\},
\]
then we have $CT_\delta(X) = \sum_{\delta < D\lesssim 1}|\mathcal L_{D,<\delta}|$, and 
\[
|\mathcal L_{D,<\delta}(X)|\le \sum_{\substack{P:\ \text{incomparable}\\\delta\times\sqrt{D\delta}\times D\ \text{lightplanks}}}|X\cap P|^2.
\]
For each $\delta<D\lesssim 1$, write $D = \delta\tau^{-2}$ for some $\sqrt\delta\le \tau\le 1$. For $\tau$ fixed, pigeonhole a value $\mu$ and a maximal pairwise incomparable collection $\mathcal P_\mu$ of $\delta\times\delta\tau^{-1}\times\delta\tau^{-2}$ lightplanks $P$ with $|X\cap P|\approx\mu$ such that
\[
\sum_{\substack{P:\ \text{incomparable}\\\delta\times\delta\tau^{-1}\times \delta\tau^{-2}\ \text{lightplanks}}}|X\cap P|^2 \approx \mu^2|\mathcal P_\mu|.
\]
Using rectangle-lightplank duality, the set
\[
\mathcal R(\mathcal P_\mu) := \{\mathbf D^*_{10\delta}(P):P\in\mathcal P_\mu\}
\]
of dual $\approx\delta,\tau$-rectangles $\Omega = \mathbf D_{10\delta}^*(P)$ as $P$ ranges in $\mathcal P_\mu$ is a maximal pairwise incomparable collection of $\approx \delta,\tau$-rectangles satisfying \eqref{eqn:rich-rectangles}. Therefore, applying rectangle-lightplank duality and Equation \eqref{eqn:rich-rectangles},
\begin{align*}
    |\mathcal L_{\delta\tau^{-2},<\delta}(X)|&\lessapprox \sum_{P\in\mathcal P_\mu}|X\cap P|^2\\
    &\le \sup\{|X\cap P|:P\ \text{is a $\delta\times\delta\tau^{-1}\times\delta\tau^{-2}$-lightplank}\}^{1/2}\sum_{P\in\mathcal P_\mu}|X\cap P|^{3/2}\\
    &\approx \mathbf P_\tau(X)^{1/2}\cdot \mu^{3/2}|\mathcal R(\mathcal P_\mu)|,
\end{align*}
where we introduce the shorthand notation
\[
\mathbf P_\tau(X) = \sup\{|X\cap P|:P\ \text{is a $\delta\times\delta\tau^{-1}\times\delta\tau^{-2}$-lightplank}\}.
\]
Now, combining Equations \eqref{eqn:pyz2} and \eqref{eqn:max-controls-rect},
\begin{align*}
|\mathcal L_{\delta\tau^{-2},<\delta}(X)|&\lessapprox \mathbf P_\tau(X)^{1/2}\cdot(\delta\tau)^{-1}\cdot\delta |X|\\
&= \mathbf P_\tau(X)^{1/2}\tau^{-1}|X|,
\end{align*}
Since $\tau\ge\sqrt\delta$, and there are $\approx 1$-many dyadic values of $\tau$,
\begin{equation}\label{eqn:answer-prob-pairs}
    CT_\delta(X) = \sum_{\sqrt\delta\le\tau\le1}|\mathcal L_{\delta\tau^{-2},<\delta}(X)|\lessapprox \mathbf P_{\sqrt\delta}(X)^{1/2}\delta^{-1/2}|X|.
\end{equation}
Inequality \eqref{eqn:answer-prob-pairs} is our answer to Problem \ref{prob:pairs} when $X$ is $1$-dimensional (see Theorem \ref{thm:num-ll-pairs} for the precise statement). The quantity $\mathbf P_{\sqrt\delta}(X)$ depends on the shape of $X$. For example, if $X$ manages to avoid lightplanks, so  $\mathbf P_{\sqrt\delta}(X) = 1$, our bound for $CT_\delta(X)$ is a significant gain over the trivial bound of $|X|^2$. See Figure \ref{fig:diff-weights} for three different 1-dimensional configurations $X$ with different values of $\mathbf P_{\sqrt \delta}(X)$.

\begin{figure}
    \centering
    \includegraphics[width=12cm]{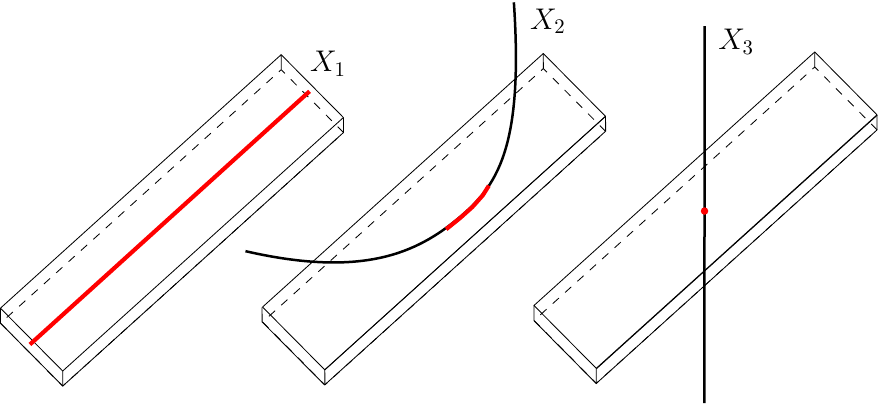}
    \caption{Three sets $X_1,X_2,X_3$ with $\mathbf P(X_1)\gg\mathbf P(X_2)\gg\mathbf P(X_3)$}
    \label{fig:diff-weights}
\end{figure}

 The rest of this section will be devoted to making the proof sketch of inequality \eqref{eqn:answer-prob-pairs} we have just presented rigorous in Theorem \ref{thm:num-ll-pairs}. To keep the exposition in this section clean, we will keep a few geometric lemmas regarding circles satisfying $\Delta(x,x')<\delta$ within Appendix \ref{appendix:rect-lightplank-lemmas}.

\subsubsection*{\textnormal{\textbf{Structure of the argument}}}

To help the reader navigate the steps that follow, we summarize the logical structure of the proof of inequality~(18). Each item corresponds to a key proposition or theorem developed in this section.

\vspace{0.5em}

\noindent\textbf{Proposition \ref{prop:change-of-basis} (Continuity of tangency).}
Shows that if one $\delta,\tau$-rectangle is contained in another, their associated dual lightplanks are comparable. This ensures that tangency sets vary continuously with small perturbations.

\par\vspace{0.75em}
\noindent\textbf{Proposition \ref{prop:incomparable} (Packing incomparable rectangles).}
Controls the number of pairwise incomparable $\delta,\tau$-rectangles that can fit inside a larger rectangle, providing a key geometric pigeonholing tool.

\par\vspace{0.75em}
\noindent\textbf{Propositions \ref{prop:fixed-x-mult}--\ref{prop:tangency-to-occupancy} (Multiplicity and occupancy).}
Relate how many circles are tangent to a rectangle (tangency multiplicity) to the maximum lightplank occupancy $\mathbf P_\tau(X)$, establishing quantitative bounds for use in counting arguments.

\par\vspace{0.75em}
\noindent\textbf{Proposition \ref{prop:main-geom} (Counting high-multiplicity rectangles).}
Uses the Pramanik--Yang--Zahl circular maximal estimate to show that rectangles with large tangency multiplicity must be rare.

\par\vspace{0.75em}
\noindent\textbf{Propositions \ref{prop:nearly-ll}--\ref{prop:both-tangent} (Geometry of nearly lightlike pairs).}
Show that nearly lightlike-separated pairs must lie in a common lightplank and hence correspond to tangency rectangles in the dual setting.

\par\vspace{0.75em}
\noindent\textbf{Theorem \ref{thm:num-ll-pairs}.}
Establishes inequality~\eqref{eqn:answer-prob-pairs} by combining the preceding geometric and counting estimates.
\par\vspace{0.75em}

Our first minor goal is to count the number of $A$-incomparable $\delta,\tau$-rectangles contained within a slightly larger rectangle in Proposition \ref{prop:incomparable}. To do so, we  first give a rigorous proof of Remark \ref{rem:continuous-construction}, which says that the tangency map $\Omega\mapsto \mathbf D_{10\delta}(\Omega)$ is ``continuous'' when $\Omega$ is a $\delta,\tau$-rectangle.

\begin{prop}[Continuity of $\mathbf D_{10\delta}$]\label{prop:change-of-basis}
Suppose $\Omega^{(v)}\subset\overline\Omega^{(w)}$, where $\Omega^{(v)}$ is a $\delta,\tau$-rectangle, and $\overline\Omega^{(w)}$ is an $A\delta,A\tau$-rectangle. Let $P^{(v)} = \mathbf D_{10\delta}(\Omega^{(v)})$ and $\overline P^{(w)} = \mathbf D_{10 A\delta}(\overline\Omega^{(w)})$. Then for an absolute constant $C>1$, $P^{(v)} \subset CA^{C} \overline P^{(w)}$.
\end{prop}
\begin{proof}
Let $\mathcal E = e_m,e_l,e_s$ and $\overline{\mathcal E} = \bar e_m, \bar e_l, \bar e_s$ be the lightlike bases associated with the lightplanks $P^{(v)}$ and $\overline P^{(w)}$, respectively. Let $\theta = \angle e_m,\bar e_m$. Because $\Omega^{(v)}\subset C_{A\delta,w}$, we have $w\in \mathbf D_{A\delta}(\Omega^{(v)})$ by Definition \ref{defn:tangency} of taking tangency.  By Proposition \ref{prop:dual-to-d,t-rectangle}, $\mathbf D_{A\delta}(\Omega^{(v)})\subset CA^CP^{(v)}$ for an appropriate large constant $C$.

Hence, by Proposition \ref{prop:change-of-coordinates}, we have
\[
\begin{matrix}
|\langle v-w,\bar e_s\rangle| & \le &  |\langle v-w, e_s\rangle| &  + & O(\theta) |\langle v-w,e_m\rangle| & + & O(\theta^2) |\langle v-w,e_l\rangle|\\
 & \lesssim  &  A^C\delta & + & O(\theta)A^C\delta\tau^{-1} &  + & O(\theta^2)A^C\delta\tau^{-2}.
\end{matrix}
\]

By Proposition \ref{prop:angle-em}, $|\theta|\lesssim A^C\tau$, so
$|\langle v-w,\bar e_s\rangle| \lesssim A^C\delta$. Analogous considerations using Proposition \ref{prop:change-of-coordinates} and $|\theta|\lesssim A^C\tau$ show $|\langle v-w,\bar e_m\rangle| \lesssim A^C\delta\tau^{-1}$ and $|\langle v-w,\bar e_l\rangle|\lesssim A^C\delta\tau^{-2}$. Since $\overline P^{(w)}$ is an $\approx \delta\times \delta\tau^{-1}\times \delta\tau^{-2}$-lightplank, this shows $v\in CA^C\overline P^{(w)}$. Now that we have shown $v\in CA^C\overline P^{(w)}$, it suffices to prove that for any $x\in P^{(v)}$, the inequalities
\begin{align*}
|\langle x-v,\bar e_s\rangle| &\lesssim A^C\delta\\
|\langle x-v,\bar e_m\rangle| &\lesssim A^C\delta\tau^{-1}\\
|\langle x-v,\bar e_l\rangle| &\lesssim A^C\delta\tau^{-2}
\end{align*}
all hold. We provide the details to estimate $|\langle x-v,\bar e_s\rangle|$ since the proofs of the remaining inequalities are entirely analogous. By Proposition \ref{prop:change-of-coordinates} again, we have
\begin{equation}\label{bar-component}
|\langle x-v,\bar e_s\rangle| \lesssim O(1)|\langle x-v,e_s\rangle| + O(\theta)|\langle x-v,e_m\rangle| + O(\theta^2)|\langle x-v,e_l\rangle|.
\end{equation}
Since $x\in P^{(v)}$, a $\approx \delta\times \delta\tau^{-1}\times\delta\tau^{-2}$-lightplank, we have
\begin{align*}
|\langle x-v,e_s\rangle| &\lesssim \delta\\
|\langle x-v,e_m\rangle| &\lesssim \delta\tau^{-1}\\
|\langle x-v,e_l\rangle| &\lesssim \delta\tau^{-2}.
\end{align*}
Substituting these bounds into \eqref{bar-component} with $|\theta|\lesssim A^C\tau$, we obtain
\[
|\langle x-v,\bar e_s\rangle| \lesssim A^C\delta.
\]
Finally, we use Proposition \ref{prop:change-of-coordinates} and $|\theta|\lesssim A^C\tau$ again to bound $|\langle x-v,\bar e_m\rangle|\lesssim A^C\delta\tau^{-1}$ and $|\langle x-v,\bar e_l\rangle|\lesssim A^C\delta\tau^{-2}$, and this finishes the proof.
\end{proof}

The next proposition is a refinement of Lemma 1.2 in Wolff's paper \cite{wolff2000local}. Since we need to be precise about the size of the final answer, and because Wolff does not include the proof, we include the proof here.

\begin{prop}[Packing incomparable rectangles]\label{prop:incomparable}
For any $A_0\ge 1$, the number of pairwise $A$-incomparable $\delta,\tau$-rectangles contained in an $A_0A\delta,A_0A\tau$-rectangle is $\lesssim (A_0A)^C$.
\end{prop}
\begin{proof}
Let $\overline\Omega^{(o)}$ be an $A_0A\delta,A_0A\tau$-rectangle, and let $\{\Omega^{(v_i)}\}_{i=1}^M$ be a maximal pairwise $A$-incomparable collection of $\delta,\tau$-rectangles contained in $\overline\Omega^{(o)}$. Let $\overline P^{(o)}\asymp \mathbf D_{10A_0A\delta}(\overline\Omega^{(o)})$ be the essentially unique $\approx \delta\times\delta\tau^{-1}\times\delta\tau^{-2}$-lightplank dual to $\overline\Omega^{(o)}$.

Let $\overline{\mathcal E}=\bar e_m,\bar e_l,\bar e_s$ be the lightlike basis associated to the lightplank $\overline P^{(o)}$. Since $\Omega^{(v_i)}\subset \overline\Omega^{(o)}$ for each $i$, by Proposition \ref{prop:change-of-basis}, we have $v_i\in (A_0A)^C\overline P^{(o)}$, so each of the following inequalities holds for every $i,j\in\{1,\dots,M\}$:
\begin{itemize}
\item $|\langle v_i-v_j,\bar e_s\rangle| \lesssim (A_0A)^C\delta$
\item $|\langle v_i-v_j,\bar e_m\rangle| \lesssim (A_0A)^C\delta\tau^{-1}$
\item $|\langle v_i-v_j,\bar e_l\rangle| \lesssim (A_0A)^C\delta\tau^{-2}$.
\end{itemize}
As the circles $v_1,\dots,v_M$ contained in $(A_0A)^C\overline P^{(o)}$ are $A$-incomparable, for each $i \ne j$, at least one of the following inequalities must hold:
\begin{itemize}
\item $|\langle v_i-v_j,\bar e_s\rangle| \gtrsim A^C\delta$
\item $|\langle v_i-v_j,\bar e_m\rangle| \gtrsim A^C\delta\tau^{-1}$
\item $|\langle v_i-v_j,\bar e_l\rangle| \gtrsim A^C\delta\tau^{-2}$.
\end{itemize}
Therefore, $M \lesssim (A_0A)^C$, and the claim is proved.
\end{proof}

For $\delta>0$ fixed, and each $\lambda\approx 1$, we recall the multiplicity function
\[
m_{\lambda\delta}(z) = \sum_{x\in X}C_{\lambda\delta,x}(z),\quad z\in\R^2.
\]
If $v\in Q$ and $\Omega^{(v)}$ is a $\delta,\tau$-rectangle, recall that $v$ is by Definition \ref{defn:tangency} $1$-tangent to $\Omega$, and that by Proposition \ref{prop:dual-to-d,t-rectangle} if $\lambda\approx 1$, the collection $\mathbf D_{\lambda\delta}(\Omega^{(v)})$ is comparable to an essentially unique $\approx\delta\times\delta\tau^{-1}\times\delta\tau^{-2}$-lightplank $P^{(v)}$.
\begin{prop}\label{prop:fixed-x-mult}
There is an absolute constant $C\ge 1$ such that the following holds. Let $\mathcal R$ be an $A$-incomparable collection of $\delta,\tau$-rectangles contained in $\bigcup_{x\in X}C_{\delta,x}$. For each $x\in X$, and each $\lambda \ge A$, let
\[
\mathcal R_{\lambda\delta}(x) = \{\Omega\in\mathcal R:x\in \mathbf D_{\lambda\delta}(\Omega)\}
\]
be the rectangles in $\mathcal R$ which are $\lambda$-tangent to $x$ (see Definition \ref{defn:tangency}). Then for every $z\in\R^2$, we have
\[
\sum_{\Omega\in\mathcal R_{\lambda\delta}(x)}\Omega(z) \lesssim \lambda^C C_{\lambda\delta,x}(z).
\]
\end{prop}
\begin{proof}
The $\delta,\tau$-rectangles $\Omega\in\mathcal R$ satisfying
\[
z\in\Omega\subset C_{\lambda\delta,x}
\]
are all contained in a $\lambda\delta,C\tau$-rectangle. The $\Omega\in\mathcal R$ are pairwise $A$-incomparable, so by Proposition \ref{prop:incomparable} on packing rectangles, the number of such $\Omega$ is at most $\lambda^C$ for some absolute constant $C$.
\end{proof}

\begin{defn}[Multiplicity of a rectangle]\label{def:mult-rect}
    For $\lambda\approx 1$, and a $\delta,\tau$-rectangle $\Omega$, let
    \[
    \mu_{\lambda\delta}(\Omega) = |X\cap \mathbf D_{\lambda\delta}(\Omega)|
    \]
    be the number of points in $X$ that are $\lambda$-tangent to $\Omega$. We refer to $\mu_{\lambda\delta}(\Omega)$ as the \emph{multiplicity} of $\Omega$ (or the \emph{$X$-multiplicity} if we want to emphasize the set $X$ of circles).
\end{defn}

We recall the notation that if $E\subset \R^2$ is a set, we may use $E(z)$ as shorthand for the indicator function $1_E(z)$. Recall the following notation we used in Proposition \ref{prop:fixed-x-mult}: if $\mathcal R$ is any set of $\delta,\tau$-rectangles, then
\[
\mathcal R_{\lambda\delta}(x) = \{\Omega\in\mathcal R: x\in\mathbf D_{\lambda\delta}(\Omega)\}
\]
is the set of rectangles of $\mathcal R$ that are $\lambda$-tangent to $x$. If $a(x,\Omega)$ is any quantity that depends on $x\in X$ and $\Omega\in\mathcal R$, then we have the following double-counting/Fubini relationship:
\[
\sum_{\Omega\in\mathcal R}\sum_{x\in X\cap \mathbf D_{\lambda\delta}(\Omega)}a(x,\Omega) = \sum_{x\in X}\sum_{\Omega\in\mathcal R_{\lambda\delta}(x)}a(x,\Omega).
\]
Before stating the next results, we briefly explain their role in the argument. Propositions~\ref{prop:maximal-mult} and~\ref{prop:tangency-to-occupancy} establish a key link between the number of circles tangent to a rectangle (tangency multiplicity) and the distribution of mass within lightplanks (occupancy). Proposition~\ref{prop:maximal-mult} connects these multiplicities to the multiplicity function $m_{\lambda\delta}$, which we can control using the Pramanik--Yang--Zahl estimate. Proposition~\ref{prop:tangency-to-occupancy} then shows that the maximum tangency multiplicity over all rectangles is comparable to the worst-case lightplank occupancy. Together, these results allow us to reduce tangency counting to estimates involving the lightplank geometry of~$X$.
\begin{prop}\label{prop:maximal-mult}
If $\mathcal R$ is a pairwise $A$-incomparable collection of $\delta,\tau$-rectangles contained in $\bigcup_{x\in X}C_{\delta,x}$ and $\lambda\ge A$, then
\[
\sum_{\Omega\in\mathcal R}\mu_{\lambda\delta}(\Omega)\Omega(z) \lesssim \lambda^Cg_{\lambda \delta}(z) = \lambda^C\sum_{x\in X}C_{\lambda\delta,x}(z),\quad z\in\R^2.
\]
\end{prop}
\begin{proof}
By Definition \ref{def:mult-rect} of $\mu_{\lambda\delta}(\Omega)$ and double-counting,
\begin{align*}
\sum_{\Omega\in\mathcal R}\mu_{\lambda\delta}(\Omega)\Omega(z) &= \sum_{\Omega\in\mathcal R}\big(\sum_{x\in X\cap\mathbf D_{\lambda\delta}(\Omega)}C_{\lambda\delta,x}(z)\big)\Omega(z) \\
&= \sum_{x\in X}C_{\lambda\delta,x}(z)\sum_{\Omega\in\mathcal R_{\lambda\delta}(x)}\Omega(z).
\end{align*}
By Proposition \ref{prop:fixed-x-mult}, for each fixed $x\in X$ and $\lambda \ge A$, the inner sum over $\Omega\in\mathcal R_{\lambda\delta}(x)$ is bounded by $\lambda^C C_{\lambda\delta,x}(z)$. This finishes the proof by the definition of $m_{\lambda\delta}(z)$.
\end{proof}

For $\delta^{1/2}\le \tau\ll 1$, let
\[
T(\tau) = \sup\{\mu_{10\delta}(\Omega):\Omega\subset\R^2\ \text{is a $\delta,\tau$-rectangle}\}
\]
be the maximal number of circles of $X$ that are $10$-tangent to a $\delta,\tau$-rectangle.
\begin{prop}[Maximal tangency number and lightplank occupancy]\label{prop:tangency-to-occupancy}
    For $\delta^{1/2}\le \tau\ll 1$, we have
    \[
    T(\tau) \approx \sup\{|X\cap P|:P\ \text{is a } \delta\times \delta\tau^{-1}\times \delta\tau^{-2}\text{-lightplank}\}.
    \]
\end{prop}
\begin{proof}
    For this proof, let $\mathbf P(\tau)$ denote the maximal $\delta\times\delta\tau^{-1}\times\delta\tau^{-2}$-lightplank occupancy of $X$:
    \[
    \mathbf P(\tau) = \sup\{|X\cap P|:P\ \text{is a } \delta\times \delta\tau^{-1}\times \delta\tau^{-2}\text{-lightplank}\}.
    \]
    Let $\Omega$ be a $\delta,\tau$-rectangle such that $\mu_{10\delta}(\Omega) = T(\tau)$. By Proposition \ref{prop:dual-to-d,t-rectangle} $\mathbf D_{10\delta}(\Omega)$ can be covered by $\approx 1$-many $\delta\times\delta\tau^{-1}\times\delta\tau^{-2}$-lightplanks. By the pigeonhole principle, there exists a $\delta\times\delta\tau^{-1}\times\delta\tau^{-2}$-lightplank $P_0$ such that
    \[
    |X\cap P_0|\gtrapprox |X\cap \mathbf D_{10\delta}(\Omega)|= \mu_{10\delta}(\Omega) = T(\tau).
    \]
    This shows that $\mathbf P(\tau)\gtrapprox T(\tau)$.
    Conversely, let $P$ be a $\delta\times\delta\tau^{-1}\times\delta\tau^{-2}$-lightplank such that $|X\cap P| = \mathbf P(\tau)$. By Proposition \ref{prop:dual-to-d,t-rectangle}, there are $\approx 1$-many $\delta,\tau$-rectangles $\Omega$ which cover $\mathbf D^*_{10\delta}(P)$. By the pigeonhole principle, there exists a $\delta,\tau$-rectangle $\Omega_0$ such that
    \[
    |X\cap \mathbf D_{10\delta}(\Omega_0)|\gtrapprox |X\cap P|=\mathbf P(\tau).
    \]
    This finishes the proof that $T(\tau)\approx\mathbf P(\tau)$.
\end{proof}
\begin{prop}\label{prop:nu-gamma}
If $\Omega$ is a $\delta,\tau$-rectangle, then for each $\lambda \approx 1$,
\[
\mu_{\lambda\delta}(\Omega)\lessapprox T(\tau).
\]
\end{prop}
\begin{proof}
This would follow immediately from the definition of $\mu_{\lambda\delta}(\Omega)$ (with ``$\le$'' in place of ``$\lessapprox$'') if $T(\tau)$ were defined with $\mathbf D_{\lambda\delta}(\Omega)$ instead of $\mathbf D_{10\delta}(\Omega)$. By Proposition \ref{prop:dual-to-d,t-rectangle}, $\mathbf D_{\lambda\delta}(\Omega)$ is comparable to a $\approx\delta\times\delta\tau^{-1}\times\delta\tau^{-2}$-lightplank, which can be covered by $\approx 1$-many translates $\{\mathbf D_{10\delta}(\Omega)+v_i\}_i$, at least one of which must satisfy by the pigeonhole principle,
\[
|X\cap (\mathbf D_{10\delta}(\Omega)+v_i)| \gtrapprox \mu_{\lambda\delta}(\Omega).
\]
Therefore, $T(\tau)\gtrapprox \mu_{\lambda\delta}(\Omega)$.
\end{proof}

For a dyadic number $1\le M\le T(\tau)$, a number $\lambda\approx 1$, and a collection $\mathcal R$ of $\delta,\tau$-rectangles, let
\[
\mathcal R_{M,\lambda} = \{\Omega\in\mathcal R:\mu_{\lambda\delta}(\Omega)\sim M\}.
\]
We are now ready to rigorously justify the proof sketch we presented at the beginning of this Section.
\begin{prop}[Counting tangency rectangles]\label{prop:main-geom}
Suppose $X\subset Q$ is a set of $\delta$-separated circles obeying the 1-dimensional Frostman non-concentration condition
\[
|X\cap B_r|\lesssim \frac{r}{\delta},\quad\text{for all $r$-balls $B_r\subset\R^3$ and $r\ge \delta$}.
\]
If $\mathcal R$ is any pairwise $A$-incomparable collection of $\delta,\tau$-rectangles contained in $\bigcup_{x\in X}C_{\delta,x}$, then for each $M\in[1,T(\tau)]$ and $A\le\lambda\lessapprox 1$,
\[
M^{3/2}|\mathcal R_{M,\lambda}| \lessapprox \tau^{-1}|X|.
\]
\end{prop}
\begin{proof}
By Proposition \ref{prop:maximal-mult}, if $\lambda \ge A$, we have 
\[
\lambda^C g_{\lambda \delta}(z)\gtrsim \sum_{\Omega\in\mathcal R}\mu_{\lambda\delta}(\Omega)\Omega(z),\quad z\in\R^2.
\]
We organize the sum on the right-hand side by the dyadic level sets of $\mu_{\lambda\delta}(\Omega)$, keeping in mind Proposition \ref{prop:nu-gamma}:
\[
\lambda^Cm_{\lambda\delta}(z) \gtrsim \sum_{\substack{1< M < T(\tau)\\M\ \text{dyadic}}}M\sum_{\Omega\in\mathcal R_{M,\lambda}}\Omega(z),\quad z\in \R^2.
\]
By Pramanik--Yang--Zahl's Theorem \ref{thm:pramanik-yang-zahl}, and the embedding $\ell^1\hookrightarrow \ell^{3/2}$,
\[
\delta |X|\gtrapprox \lambda^{C}\int_{\R^2} m_{\lambda\delta}(z)^{3/2}\,dz\gtrsim M^{3/2}|\mathcal R_{M,\lambda}|\cdot|\Omega|.
\]
Dividing by $|\Omega| \sim \delta\tau$ finishes the proof.
\end{proof}

\subsection{Nearly lightlike separated pairs}\label{subsec:pairs}
For dyadic numbers $0<\Delta\le D<1$, define the collection
\[
\mathcal L_{D,\Delta}(X) = \{(v,w)\in X\times X : d(v,w) \sim D, \Delta(v,w) \sim \Delta\}.
\]
We will refer to a pair $(v,w)\in\mathcal L_{D,\Delta}(X)$ as \emph{nearly lightlike separated} when $\Delta\lessapprox\delta$.  A pair of nearly lightlike separated points is the same as a pair of nearly internally tangent circles in terms of point-circle duality. Let $\tau_D = \delta^{1/2}D^{-1/2}$, and recall
\[
T(\tau_D) = \sup\{|X\cap \mathbf D_{10\delta}(\Omega)|:\Omega\subset\R^2\ \text{is a $\delta,\tau_D$-rectangle}\}
\]
is the maximal number of circles of $X$ that are $10$-tangent to any $\delta,\tau_D$-rectangle. The reason for the definition $\tau_D = \delta^{1/2}D^{-1/2}$ is that if $d(v,w)\approx D$, then both $v$ and $w$ belong to an $\approx \delta\times\sqrt{\delta D}\times D$-lightplank, and the circles $v,w$ are $\gtrapprox \delta,\tau_D$-tangent, in the following sense.

\begin{defn}\label{defn:delta-tau-tangent}
We say two circles $v,w$ are $\gtrapprox\delta,\tau$-tangent if there are $\approx1$-comparable $\delta,\tau$-rectangles $\Omega^{(v)}\subset C_{\delta,v}$, $\Omega^{(w)}\subset C_{\delta,w}$.
\end{defn}

\begin{prop}\label{prop:nearly-ll}
If $(v,w)\in \mathcal L_{D,\Delta}(X)$ for $D\gg\delta$ and $\Delta\lessapprox\delta$, then $v,w$ are $\gtrapprox\delta,\tau_D$-tangent.
\end{prop}

\begin{proof}
Suppose $D\gg\delta$, $\Delta\lessapprox\delta$, and $(v,w)\in\mathcal L_{D,\Delta}(X)$. We will find a $\approx\delta\times\delta\tau^{-1}\times\delta\tau^{-2}$-lightplank $P$ such that both $v,w\in P$. By duality, for appropriate $\lambda\approx 1$,
\[
\mathbf D^*_{\lambda\delta}(P)\cap C_{\delta,v}\ \text{and}\ \mathbf D_{\lambda\delta}^*(P) \cap C_{\delta,w}
\]
 are $\approx 1$-comparable $\approx \delta,\tau_D$-rectangles contained in $C_{\delta,v}$ and $C_{\delta,w}$, respectively, so this will finish the proof.

Recall $\Gamma_v$ is the lightcone with vertex $v$. Let $w_0\in \Gamma_v$ be the nearest point to $w$. By definition, $v-w_0$ is a lightlike vector, and since $\delta\ll D$, we have $|v-w_0|\sim |v-w|\sim D$ and $|w_0-w|\sim\Delta(v,w)\lessapprox\delta$. Choose a unit vector $e_l$ parallel to $v-w_0$, another unit vector $e_m\in\R^2\times\{0\}$ orthogonal to $\pi_{\R^2}(v-w_0)$ (the projection of $v-w_0$ to $\R^2$), and a third unit vector $e_s$ so that $e_m,e_l,e_s$ is a lightlike basis for $\R^3$. It suffices to check
\begin{itemize}
\item[(i)] $|\langle v-w,e_l\rangle|\lessapprox \delta\tau_D^{-2}$,
\item[(ii)] $|\langle v-w,e_m\rangle|\lessapprox \delta\tau_D^{-1}$, and
\item [(iii)] $|\langle v-w,e_s\rangle|\lessapprox \delta$.
\end{itemize}
Since $\delta\tau_D^{-2} = D\sim |\langle v-w,e_l\rangle|$, and $\Delta(v,w)=|w-w_0|\sim|\langle v-w,e_s\rangle|\lessapprox \delta$, only point (ii) needs elaboration. But by elementary geometry considerations, this is a simple consequence of the assumption $d(v,w)\sim D$ and $\Delta(v,w)\lessapprox \delta$.
\end{proof}

\begin{prop}[Covering by lightplanks]\label{prop:both-tangent}
There is an absolute constant $C>1$ so that the following holds. If $(v,w)\in\mathcal L_{D,\Delta}(X)$ and $\mathcal R$ is a maximal pairwise $A$-incomparable collection of $\delta,\tau_D$-rectangles contained in $\bigcup_{x\in X}C_{\delta,x}$, then there exists $\Omega_0\in\mathcal R$ so that $v,w\in\mathbf D_{CA^C\delta}(\Omega_0)$.
\end{prop}

\begin{proof}
By Proposition \ref{prop:nearly-ll}, there are $\approx 1$-comparable $\delta,\tau_D$-rectangles $\Omega^{(v)},\Omega^{(w)}$ in $C_{\delta,v},C_{\delta,w}$ respectively. By maximality of $\mathcal R$ with respect to $A$-incomparability, there is some $\Omega_0\in\mathcal R$ such that $\Omega^{(v)}$ (say) is comparable to $\Omega_0$, hence $v\in\mathbf D_{CA^C}(\Omega_0)$ and since $\Omega^{(v)}$ and $\Omega^{(w)}$ are comparable, by almost-transitivity (Proposition \ref{almost-transitivity}), $\Omega_0$ and $\Omega^{(w)}$ are $CA^C$-comparable. Hence $w\in\mathbf D_{CA^C\delta}(\Omega_0)$ for a large enough absolute constant $C$, and the claim is proved.
\end{proof}

\begin{theorem}[Nearly lightlike separated pairs for a $1$-dimensional family of circles]\label{thm:num-ll-pairs}
    Suppose $X\subset Q=B(e_3,\alpha_0)$ is a set of $\delta$-separated circles obeying the 1-dimensional Frostman non-concentration condition
    \[
    |X\cap B_r|\lesssim \frac{r}{\delta},\quad\text{for all $r$-balls $B_r\subset\R^3$ and $r\ge \delta$}.
    \]
    If $\Delta\lessapprox \delta$ and $D \gg \delta$, then $|\mathcal L_{D,\Delta}(X)| \lessapprox T(\tau_D)^{1/2}(\delta^{-1}D)^{1/2}|X|$, where $\tau_D = \delta^{1/2}D^{-1/2}$.
\end{theorem}
\begin{proof}
Let $A\approx 1$ be a parameter (take $A = \delta^{-\epsilon}$ for definiteness), and fix an arbitrary maximal pairwise $A$-incomparable collection $\mathcal R$ of $\delta,\tau_D$-rectangles contained in $\bigcup_{x\in X}C_{\delta,x}$. 

By Proposition \ref{prop:both-tangent}, for a given $(v,w)\in \mathcal L_{D,\Delta}(X)$, we can find a rectangle $\Omega\in\mathcal R$ such that $v,w\in \mathbf D_{\lambda\delta}(\Omega)$ for some $\lambda=A^{O(1)}$, and we can write
\[
\mathcal L_{D,\Delta}(X)\subset \bigcup_{\Omega\in\mathcal R}\{(v,w)\in X\times X:v,w\in\mathbf D_{\lambda\delta}(\Omega)\}.
\]
By the union bound,
\begin{equation}\label{union}
|\mathcal L_{D,\Delta}(X)| \le \sum_{\Omega\in\mathcal R}|X\cap \mathbf D_{\lambda\delta}(\Omega)|^2 = \sum_{\Omega\in\mathcal R}\mu_{\lambda\delta}(\Omega)^2.
\end{equation}

Recall that by Proposition \ref{prop:nu-gamma}, $\mu_{\lambda\delta}(\Omega)\lessapprox T(\tau_D)$. We organize the last sum on the right-hand side of \eqref{union} by the dyadic value of $\mu_{\lambda\delta}(\Omega)$, up to $T(\tau_D)$. Letting $\mathcal R_{M,\lambda} = \{\Omega\in\mathcal R:\mu_{\lambda\delta}(\Omega)\sim M\}$, we estimate \eqref{union} by
\[
\sum_{\substack{1<M<T(\tau_D)\\M\ \text{dyadic}}}M^2|\mathcal R_{M,\lambda}| \le T(\tau_D)^{1/2}\sum_{1<M<T(\tau_D)}M^{3/2}|\mathcal R_{M,\lambda}|.
\]
By Proposition \ref{prop:main-geom}, for each $M$, $M^{3/2}|\mathcal R_{M,\lambda}|\lessapprox \tau_D^{-1}|X| = (\delta^{-1}D)^{\frac12}|X|$. As $T(\tau_D)\le |X\cap Q|\lesssim\delta^{-1}$ for any $1$-dimensional set of $\delta$-separated circles in $Q=B(e_3,\alpha_0)$, there are $\approx 1$-many values of $M$ in the sum, so we have shown $|\mathcal L_{D,\Delta}(X)|\lessapprox T(\tau_D)^{1/2}(\delta^{-1}D)^{1/2}|X|$. This finishes the proof.
\end{proof}

\section{Proof of main theorems}\label{sec:main-thm}

In addition to the estimate of circle tangencies for a 1-dimensional set of circles, we need a pointwise bound for the  Fourier transform of $d\sigma$, a smooth surface carried measure on the cone segment. We state the version of the estimate we will use in the proof of Theorem \ref{thm:ortiz-sparse-l1} here. The proof of this lemma is contained in Appendix \ref{app:ft}.
\begin{lemma}\label{lem:fourier-decay}
Let $d\sigma$ be a smooth surface measure supported in $\mathbb Cone^2$. For any $\epsilon>0$, there is a constant $C_\epsilon$ so that
\[
|\widecheck{d\sigma}(x)|\le C_\epsilon\frac1{(1+|x|)^{\frac12-\epsilon}}\frac1{(1+d(x,\Gamma_0))^{100\epsilon^{-1}}}
\]
holds for all $x\in \R^3$, where $\Gamma_0 = \{(a,r)\in\R^2\times\R:||a|-|r||=0\}$ is the lightcone with vertex $0$.
\end{lemma}

Now we are ready to give the proof of Theorem \ref{thm:ortiz-sparse-l1}, whose statement we recall here.
\begin{retheorem}
For each $\epsilon > 0$,  there is a constant $C_\epsilon$ so the following holds for each $R > 1$. Suppose $X\subset B_R$ is a disjoint union of unit balls that satisfies the 1-dimensional Frostman non-concentration condition
\[
|X\cap B(x,r)|\lesssim  r,\quad x\in \R^3,r>1.
\]
Let $\mathbf P(X)$ be the quantity
\[
\mathbf P(X) = \sup\{|X\cap P|:P\ \text{is a lightplank of dimensions $1\times R^{1/2}\times R$}\}.
\]
Then for any measurable function $w\colon\R^3\to[0,1]$ and any $f\in L^2(\mathbb Cone^2,d\sigma)$, the estimate
\[
\int_X|E_{\mathbb Cone^2}f|w\le C_\epsilon R^\epsilon\,\mathbf P(X)^{1/4}(\int_Xw)^{1/2}\|f\|_{L^2(d\sigma)}
\]
holds.
\end{retheorem}
\begin{proof}
    For each dyadic $\eta\le 1$, let $X_\eta$ be the union of unit balls $Q\subset X$ satisfying $\int_Qw\sim\eta$. The sets $X_\eta$ are pairwise disjoint, and we can write
\[
\int_X|Ef|w = \sum_\eta \int_{X_\eta} |Ef|w.
\]
Moreover,
\[
\int_Xw = \sum_\eta\int_{X_\eta}w=\sum_\eta\sum_{Q\subset X_\eta}\int_Qw\sim\sum_\eta\eta|X_\eta|.
\]
Set $d\mu_\eta = w1_{X_\eta}dx$. By duality, for an appropriate $|g|\le 1$,
\[
\int |Ef|\,d\mu_{\eta} = \int Ef\cdot g\,d\mu_\eta. 
\]
By Plancherel and Cauchy--Schwarz, this is bounded by
\[
(\int_{\mathbb Cone^2}|\widehat{g\mu_\eta}|^2\,d\sigma)^{1/2}\|f\|_{L^2(d\sigma)}.
\]
Next, we use Fourier transform properties, $|g|\le 1$, and the triangle inequality to bound this by
\[
(\iint_{X_\eta\times X_\eta}|\widecheck{d\sigma}(x-y)|w(x)w(y)\,dx\,dy)^{1/2}\|f\|_{L^2(d\sigma)}.
\]
By dyadic pigeonholing and Lemma \ref{lem:fourier-decay}, up to a rapidly decaying tail, for some $1<\rho<R$ this is bounded by
\[
(\frac{1}{\rho^{1/2}} \sum_{\substack{Q_1,Q_2\subset X_\eta\\d(Q_1,Q_2)\sim \rho\\\Delta(Q_1,Q_2)<R^\epsilon}}\iint_{Q_1\times Q_2}w(x)w(y)\,dxdy )^{1/2}\|f\|_{L^2(d\sigma)}.
\]
By the definition of $X_\eta$, $\iint_{Q_1\times Q_2}w(x)w(y
)\,dxdy\sim\eta^2$, so this is equal to
\[
\Big(\frac{\eta^2}{\rho^{1/2}}\#\big\{(Q_1,Q_2)\subset X_\eta\times X_\eta:d(Q_1,Q_2)\sim \rho,\Delta(Q_1,Q_2)<R^\epsilon\big\}\Big)^{1/2}\|f\|_{L^2(d\sigma)}.
\]
Since $X_\eta$ is a disjoint union of unit balls satisfying $|X_\eta\cap B_r|\le|X\cap B_r|\lesssim r$ for all $r$-balls and $r\ge 1$, we use Pramanik--Yang--Zahl's maximal estimate in the form of Theorem \ref{thm:num-ll-pairs} (after rescaling the ball of radius $R$ to a ball of radius $1$) to bound the previous displayed expression by
\[
\big(\frac{\eta^2}{\rho^{1/2}}\cdot \rho^{1/2}\mathbf P(X_\eta)^{1/2}|X_\eta|\big)^{1/2}\|f\|_{L^2(d\sigma)}.
\]
Simplifying and using $\mathbf P(X_\eta)\le\mathbf P(X)$, we arrive at
\[
\int_{X_\eta}|Ef|w\lessapprox \mathbf P(X)^{1/4}\cdot \eta |X_\eta|^{1/2}\|f\|_{L^2(d\sigma)}.
\]

Summing over $\eta>R^{-500}$ gives
\[
\int_X|Ef|w \lessapprox \mathbf P(X)^{1/4}(\sum_{\substack{\eta>R^{-500}\\\eta\ \text{dyadic}}} \eta|X_\eta|^{1/2})\|f\|_{L^2(d\sigma)} + O(R^{-100})\|f\|_{L^2(d\sigma)}.
\]
There are $\approx 1$ summands, so using Cauchy--Schwarz and neglecting the error term gives
\[
\int_X|Ef|w\lessapprox \mathbf P(X)^{1/4}(\sum_{\substack{\eta>R^{-500}\\\eta\ \text{dyadic
}}}\eta^2|X_\eta|)^{1/2}\|f\|_{L^2(d\sigma)}.
\]
Finally, $\eta\le 1$, and since $\sum_\eta\eta|X_\eta|\sim \int_Xw$, we have arrived at the desired bound
\[
\int_X|Ef|w\lessapprox \mathbf P(X)^{1/4}(\int_Xw)^{1/2}\|f\|_{L^2(d\sigma)}.
\]
\end{proof}
By writing $\int_{\mathbb Cone^2}|\widehat\mu|^2\,d\sigma = \iint_{X\times X}\widecheck{d\sigma}(x-y)\,dx\,dy$ and then repeating the part of the proof of Theorem \ref{thm:main-l2-sparse} that deals with this kind of expression for $w=1$, we obtain the following refinement of $\gamma_3(1)=1/2$ described in the introduction (recall Definition \ref{def:cone-rate}):
\begin{cor}[Refinement of $\gamma_3(1)=1/2$]
    For each $\epsilon > 0$, there is a constant $C_\epsilon$ so the following holds for each $R > 1$. Suppose $X\subset B_R$ is a 1-dimensional disjoint union of unit balls, and let $d\mu = 1_Xdx$.
    Let $\mathbf P(\mu)$ be the quantity
    \[
    \mathbf P(\mu) = \sup\{\mu(P):P\ \text{is a $1\times R^{1/2}\times R$-lightplank}\}.
    \]
    Then the estimate
    \[
    \int_{\mathbb Cone^2} |\widehat\mu|^2d\sigma \le C_\epsilon R^{\epsilon}\, \mathbf P(\mu)^{1/2}\mu(B_R)
    \]
    holds.
\end{cor}

\subsection{From Theorem \ref{thm:ortiz-sparse-l1} to Theorem \ref{thm:main-l2-sparse}} For each $1\le q < \infty$ and measurable $w\colon\mathbb R^n\to[0,\infty)$, let $S_q(\mathcal M,w)$ be the smallest constant so that
\[
(\int |E_{\mathcal M}f|^qw)^{1/q}\le S_q(\mathcal M,w)\|f\|_{L^2(\mathcal M)}
\]
holds for all $f\in L^2(\mathcal M)$. If $w = 1_X$, the indicator function of a set $X$, then we write $S_q(\mathcal M,X)$ instead of $S_q(\mathcal M,1_X)$. When $q = 1$, we have the following Proposition:
\begin{prop}[From Fourier averages to weighted Fourier extension estimates]\label{prop:fourier-mean-to-sparse-l1}
    We have
    \[
    S_1(\mathcal M,X) \le \sup_{\|h\|_{L^\infty(X)}=1}(\int_{\mathcal M}|\widehat{h}|^2\,d\sigma)^{1/2}.
    \]
\end{prop}
\begin{proof}
    By duality, $\int_X |Ef| = \sup_{\|h\|_{L^\infty(X)}=1}\int_X Ef\cdot h$. By Plancherel and Cauchy--Schwarz, the claim follows.
\end{proof}
At first glance, the estimate of $S_1(\mathcal M,X)$ in terms of Fourier averages in Proposition \ref{prop:fourier-mean-to-sparse-l1} may not seem so useful since we are interested in $S_2(\mathcal M,X)$, and by H\"older's inequality, we have $S_1(\mathcal M,X)\le |X|^{1/2}S_2(\mathcal M,X)$. However, the inequality $S_1(\mathcal M,X)\le |X|^{1/2}S_2(\mathcal M,X)$ can essentially be reversed:
\begin{theorem}\label{thm:transfer}
    Let $X$ be a disjoint union of $N$ unit balls, and let $1< q<\infty$. Let $q' = \frac{q}{q-1}$ be the H\"older conjugate exponent of $q$. Suppose that for some $A>1$ (which is allowed to depend on $X$) and for every measurable function $w\colon\R^n\to [0,1]$, we know
    \[
    \int_X |Ef|w \le A(\int_Xw)^{\frac{1}{q'}}\|f\|_{L^2(\mathcal M)}.
    \]
    Then for every $f$ with $\|f\|_{L^2(\mathcal M)}=1$,
    \[
    (\int_X|Ef|^q)^{\frac 1q}\lesssim A(\log N)^{\frac1q} .
    \]
    In other words, if $S_1(\mathcal M,w1_X)\le A(\int w1_X)^{1/q'}$ for every $w\colon\R^n\to[0,1]$, then $S_q(\mathcal M,X)\lesssim A(\log N)^{1/q}$.
\end{theorem}
\begin{remark}
    Theorem \ref{thm:transfer} is closely related to Lemma C.1 in Barcel\'o--Bennett--Carbery--Rogers' work on the dimension of divergence sets for dispersive equations \cite{barcelo2011dimension}.
\end{remark}
\begin{proof}
    Let $\|f\|_{L^2(\mathcal M)} = 1$, and let $c>0$ be very small to be determined. Note $|Ef|\le \|f\|_{L^1(\mathcal M)}\lesssim \|f\|_{L^2(\mathcal M)}=1$, so we can write
    \begin{align*}
    \int_X|Ef|^q &= \int_{X\cap \{|Ef|<c\}}|Ef|^q + \sum_{c<\lambda\lesssim 1}\int_{X\cap\{|Ef|\sim\lambda\}}|Ef|^q
    \end{align*}
    where the sum is over dyadic values of $\lambda$. For the first term we can bound it by
    \[
    \int_{X\cap \{|Ef|<c\}}|Ef|^q \le c^qN < 1,
    \]
    if $c<N^{-1/q}$. If this term dominates, we are done, so we assume the second term dominates. We write out the second term as
    \begin{align*}
    \sum_{c<\lambda\lesssim1}\int_{X\cap\{|Ef|\sim\lambda\}}|Ef|^q&\sim \sum_{c<\lambda<1}\lambda^{q-1}\int_X|Ef|1_{\{|Ef|\sim \lambda\}}.
    \end{align*}
    By assumption, since $0\le 1_{\{|Ef|\sim\lambda\}}\le 1$ and $\|f\|_{L^2}=1$, this is bounded by
    \[
    A\sum_{c<\lambda\lesssim 1}\lambda^{q-1}|X\cap\{|Ef|\sim \lambda\}|^{\frac{q-1}{q}}.
    \]
    By H\"older's inequality, this is bounded by
    \begin{align*}
    A(\log N)^{\frac1q}\big(\sum_{c<\lambda\lesssim 1}\lambda^q|X\cap\{|Ef|\sim \lambda\}|\big)^{\frac{q-1}{q}}
    &\lesssim A(\log N)^{\frac1q}(\int_X|Ef|^q)^{\frac{q-1}{q}}.
    \end{align*}
    Finally, we cancel $(\int_X|Ef|^q)^{\frac{q-1}{q}}$ from both sides to obtain
    \[
    (\int_X|Ef|^q)^{\frac{1}{q}}\lesssim A(\log N)^{\frac{1}{q}}.
    \]
\end{proof}
By Theorem \ref{thm:transfer}, Theorem \ref{thm:main-l2-sparse} holds as a corollary of Theorem \ref{thm:ortiz-sparse-l1}.
\begin{cor}
    Theorem \ref{thm:main-l2-sparse} holds.
\end{cor}

Lastly, we describe examples that establish the sharpness of Theorem \ref{thm:main-l2-sparse}.
\begin{theorem}\label{thm:sharp}
For each $R>1$, and each $T\in[1,R]$, there is a nonzero function $f\in L^2(\mathbb Cone^2)$ and disjoint union of unit balls $X\subset B(0,R)$ satisfying the 1-dimensional Frostman condition
\[
|X\cap B(x,r)|\lesssim r,\quad x\in \R^3,r>1,
\]
such that $\mathbf P(X)\sim T$, and
\[
\int_X |E_{\mathbb Cone^2}f|^2 \gtrapprox T^{1/2}\|f\|_{L^2(d\sigma)}^2.
\]
\end{theorem}
\begin{proof}
Let $f = 1_\theta$, where
\[
\theta = [1,\frac{3}{2}]\times[0,\frac{T^{-1/2}}{1000}]\subset \{\xi\in \R^2: 1<|\xi|<2\}.
\]
Let $e_m,e_l,e_s$  be our favorite lightlike basis (see Remark \ref{rem:favorite-posn}), and set
\[
P = \{x_me_m + x_le_l + x_se_s:|x_m|\le T^{1/2}, |x_l|\le T, |x_s|\le 1\}.
\]
By definition of $f$, $E_{\mathbb Cone^2}f(x)\gtrsim |\theta|\sim T^{-1/2}$ for $x\in P$. (Note $Ef(x) = |Ef(x)|$ for $x\in P$.)

Let $X$ be any $1\times 1 \times T$-tube contained in the lightplank $P$. By construction, $X$ is 1-dimensional and $\mathbf P(X)\gtrsim |X\cap P|\sim T$. Also, $\|f\|_{L^2(d\sigma)}^2=|\theta|\sim T^{-1/2}$, and
\[
\int_X |Ef|^2 \gtrsim T^{-1}|X\cap P| \sim T^{1/2}\|f\|_{L^2(d\sigma)}^2,
\]
as desired.
\end{proof}

\section{Refined decoupling and the Mizohata--Takeuchi conjecture}\label{sec:discuss}
The current best progress on Conjecture \ref{conj:mt-intro} for spheres is based on an application of refined decoupling estimates (see Theorem 4.4 of \cite{guth2020falconer}), and is due to Carbery, Iliopoulou, and Wang \cite{carbery2024some}. The special case of Carbery, Iliopoulou, and Wang's main estimate for spheres when $w$ is the indicator function of a disjoint union of unit balls is the following:

\begin{theorem}[Carbery--Iliopoulou--Wang]\label{thm:mt-recent}
    Let $S$ be the unit sphere in $\R^n$. For each $\epsilon > 0$,  there is a constant $C_\epsilon$ so the following holds for each $R > 1$. Suppose $X\subset B_R$ is an arbitrary disjoint union of unit balls, and
    let $\mathbf P(X)$ be the quantity
    \[
    \mathbf P(X) = \sup\{|X\cap P|:P\ \text{is a $R^{1/2}\times\dots\times R^{1/2}\times R$-tube (with any orientation)}\}.
    \]
    Then for all $f\in L^2(S)$, the estimate
    \begin{equation}\label{eqn:ciw}
    \int_X |E_Sf|^2\le C_\epsilon R^\epsilon\, \mathbf P(X)^{\frac{2}{n+1}}\|f\|_{L^2(S)}^2
    \end{equation}
    holds.
\end{theorem}
Theorem \ref{thm:mt-recent} is sharp by letting $X$ be a set $P$ as in the definition of $\mathbf P(X)$, and letting $f$ be an appropriate Knapp example.

The sets $P$ that are $R^{1/2}\times\dots\times R^{1/2}\times R$-tubes with directions orthogonal to $S$ in the definition of $\mathbf P(X)$ have the same shape as the supports of wave packets in the wave packet decomposition of $E_Sf$. As a consequence of Theorem \ref{thm:mt-recent}, by covering one of the sets $P$ with about $R^{\frac{n-1}{2}}$-many $1\times\dots\times 1\times R$-tubes, Conjecture \ref{conj:mt-intro} holds with an $R^{\frac{n-1}{n+1}}$ loss. Carbery--Iliopoulou--Wang show that it is impossible to improve on the $R^{\frac{n-1}{n+1}}$ loss using decoupling in a precise sense. They prove their partial progress on Conjecture \ref{conj:mt-intro} for spheres is sharp given the decoupling axioms of Guth \cite{HAPPY}, which are an axiomatization of the two key ingredients used to prove Bourgain--Demeter's celebrated decoupling theorem \cite{bourgain2015proof}, namely local constancy, and local $L^2$ orthogonality.

Theorem \ref{thm:main-l2-sparse} is a sharp weighted Fourier extension estimate for the cone in $\R^3$ and the $1$-dimensional weights. By covering a $1\times R^{1/2}\times R$-lightplank with about $R^{1/2}$-many $1\times 1\times R$-tubes whose directions are orthogonal to $\mathbb Cone^2$, we deduce Corollary \ref{cor:mt-progress}: the Mizohata--Takeuchi conjecture for the cone and the 1-dimensional weights holds with a power $R^{1/4}$ loss. I would like to compare this with what we can achieve for $\mathbb Cone^2$ and general sets using refined decoupling.

By applying refined decoupling estimates for $\mathbb Cone^2$ (see Theorem A.1 of \cite{harris2022length} for a precise statement of one version of refined decoupling theorem for the cone), we have the following sharp weighted Fourier extension theorem for $\mathbb Cone^2$ and arbitrary disjoint unions of unit balls. I thank Ciprian Demeter for allowing me to include the following theorem and a sketch of its proof in this article.
\begin{theorem}[Ciprian Demeter, 2023---private communication]\label{thm:ciprian}
    For each $\epsilon>0$ there is a constant $C_\epsilon$ so the following holds for every $R>1$. Let $X\subset B_R$ be an arbitrary disjoint union of unit balls.
    
    Let $\mathbf P(X)$ be the quantity
    \[
    \mathbf P(X) = \sup\{|X\cap P|:P\ \text{is a $1\times R^{1/2}\times R$-lightplank}\}.
    \]
    Then for every $f\in L^2(\mathbb Cone^2)$,
    \begin{equation}\label{eqn:cone-ciprian}
    \int_X |E_{\mathbb Cone^2}f|^2 \le C_\epsilon R^\epsilon\,\mathbf P(X)^{2/3}\|f\|_{L^2(\mathbb Cone^2)}^2.
    \end{equation}
\end{theorem}
\begin{proof}[Proof sketch]
Let $Ef := E_{\mathbb Cone^2}f$ be the Fourier extension operator for the truncated cone, and apply the wave packet decomposition of $Ef$ at scale $R$, writing
\[
Ef = \sum_{P\in \mathcal{P}} a_P \psi_P,
\]
where each $\psi_P$ is essentially supported on a $1 \times R^{1/2} \times R$-lightplank $P$, oriented in a direction orthogonal to an angular cap on the cone of width $R^{-1/2}$. (See Chapter 2 of \cite{demeter2020fourier} for details.)

\medskip
\noindent
\textit{Simplifying assumptions.} For clarity, we ignore Schwartz tails and assume $|\psi_P| = 1_P$ and $a_P = 1$. This harmless simplification allows us to work directly with the planks in the collection $\mathcal P$. The interested reader may consult \cite{carbery2024some} for a fully rigorous version of the analogous argument for the sphere.

\medskip
\noindent
By conservation of mass and orthogonality of wave packets:
\[
R\|f\|_{L^2}^2\sim\|Ef\|_{L^2(B_R)}^2 \sim \sum_{P \in \mathcal{P}} \|\psi_P\|_{L^2(B_R)}^2 \sim \sum_{P\in\mathcal P} |P| \sim R^{3/2} |\mathcal{P}|.
\]
Hence,
\begin{equation} \label{eq:L2count}
\|f\|_{L^2}^2 \sim R^{1/2} |\mathcal{P}|.
\end{equation}

\medskip
\noindent
By Hölder’s inequality,
\begin{equation} \label{eq:Holder}
\int_X |Ef|^2 \le |X|^{2/3} \left( \int_X |Ef|^6 \right)^{1/3}.
\end{equation}

\medskip
\noindent
By dyadic pigeonholing, we may assume each unit ball in $X$ is intersected by approximately $r$ planks from $\mathcal{P}$. Using refined decoupling for the truncated cone,
\[
\|Ef\|_{L^6(X)} \lessapprox r^{1/3} \left( \sum_{P \in \mathcal{P}} \|\psi_P\|_{L^6(\mathbb{R}^3)}^6 \right)^{1/6} \sim r^{1/3} |\mathcal{P}|^{1/6} R^{1/4}.
\]
Substitute into \eqref{eq:Holder}:
\begin{equation} \label{eq:precount}
\int_X |Ef|^2 \lessapprox |X|^{2/3} r^{2/3} |\mathcal{P}|^{1/3} R^{1/2}.
\end{equation}

\medskip
\noindent
Now relate $r$ and $\mathbf{P}(X)$. Each unit ball is intersected by $\sim r$ planks, and each plank intersects at most $\mathbf{P}(X)$ unit balls. Thus, the total number of incidences satisfies:
\[
r |X| \lesssim (\log R) \cdot \mathbf{P}(X) |\mathcal{P}|.
\]
Solve for $r$ and plug into \eqref{eq:precount}:
\[
\int_X |Ef|^2 \lessapprox |X|^{2/3} \left( \frac{\mathbf{P}(X) |\mathcal{P}|}{|X|} \right)^{2/3} |\mathcal{P}|^{1/3} R^{1/2}.
\]
Simplifying:
\[
\int_X |Ef|^2 \lessapprox \mathbf{P}(X)^{2/3} |\mathcal{P}| R^{1/2}.
\]

\medskip
\noindent
Finally, apply \eqref{eq:L2count}, giving:
\[
\int_X |Ef|^2 \lessapprox \mathbf{P}(X)^{2/3} \|f\|_{L^2}^2.
\]
The claim follows.
\end{proof}

A sharp example for the estimate \eqref{eqn:cone-ciprian} is obtained by taking $X$ to be a single $1 \times R^{1/2} \times R$-lightplank and choosing $f$ to be a standard Knapp-type example as in Theorem~\ref{thm:sharp}. Since a light plank can be covered by roughly $R^{1/2}$-many $1 \times 1 \times R$-tubes, Theorem~\ref{thm:ciprian} implies a version of Conjecture~\ref{conj:mt-intro} for the cone and for unions of unit balls, with an $R^{1/3}$ loss. This is slightly worse than the $R^{1/4}$ loss we prove in Corollary~\ref{cor:mt-progress} for 1-dimensional unions of unit balls.

In the work of Carbery–Iliopoulou–Wang \cite{carbery2024some}, the authors demonstrate that the $R^{1/3}$ loss in their Mizohata–Takeuchi-type estimate for the unit circle $S$ in the plane cannot be improved using Guth’s decoupling axioms alone. They construct an ``enemy scenario'' that satisfies the decoupling axioms but does not arise from any function of the form $E_S f$ \cite{HAPPY}. It would be very interesting to determine whether a similar enemy scenario exists for the cone. Such a construction would imply that the $R^{1/3}$ loss in the Mizohata–Takeuchi conjecture for the cone in $\mathbb{R}^3$, for general sets $X$, cannot be improved without exploiting structural properties of the extension operator $E_{\mathbb{C}one^2}$ beyond those encoded in the decoupling axioms.

\appendix
\section{Lemmas of circle geometry}\label{appendix:rect-lightplank-lemmas}
In this Appendix we record a few facts about the geometry of overlapping circles and their corresponding dual lightplanks for use in Section \ref{sec:pc-duality}. The first Lemma describes the region of intersection of two $\delta$-annuli: generally two $\delta$-annuli intersect in two $\delta,\tau$-rectangles for some $\tau$ depending on the distance and degree of internal tangency.

We recall our notation from Remark \ref{rem:notation} that $Q = B(e_3,\alpha_0)$ for a small absolute constant ($\alpha_0 = 1/100$ will do).
\begin{lemma}[Wolff \cite{wolff1999recent}, Lemma 3.1 (a)]\label{lemma:annuli-wolff}
Assume that $v,w$ are two circles in $Q$. Let $d = d(v,w)$, and $\Delta = \Delta(v,w)$. Then
\[
|C_{\delta,v}\cap C_{\delta,w}|\lesssim \delta\cdot \frac{\delta}{\sqrt{(d+\delta)(\Delta+\delta)}}.
\]
\end{lemma}
\begin{cor}\label{cor:tau-estimate}
If $\Omega\subset C_{\delta,v}\cap C_{\delta,w}$ is a $\delta,\tau$-rectangle, then
\[
\tau \lesssim \frac{\delta}{\sqrt{(d(v,w)+\delta)(\Delta(v,w)+\delta)}}
\]
\end{cor}
\begin{proof}
Let $d = d(v,w)$, $\Delta = \Delta(v,w)$. By Lemma \ref{lemma:annuli-wolff},
\[
\delta \tau \sim |\Omega|\le |C_{\delta,v}\cap C_{\delta,w}| \lesssim \delta\cdot\frac{\delta}{\sqrt{(d+\delta)(\Delta+\delta)}}.
\]
Canceling $\delta$ from both sides of the inequality gives the desired result.
\end{proof}

The next Proposition calculates the angle of intersection of two circles in terms of their distance and degree of internal tangency. The angle of intersection of two circles is a useful quantity when comparing lightlike coordinates. See Proposition \ref{prop:change-of-basis} for our use of it in the proof of continuity of the tangency map.

\begin{prop}\label{prop:intersect-angle}
Suppose $C_1,C_2$ are two circles in $Q$ (i.e., two points in $Q$) which intersect at a point $a\in\R^2$. Let $u_1,u_2$ be the positively oriented unit tangent vectors to $C_1,C_2$, respectively at the point $a$. Then $\angle u_1,u_2\sim \sqrt{d(C_1,C_2)\Delta(C_1,C_2)}$
\end{prop}
\begin{proof}
Without loss of generality, suppose $C_1 = (0,0,r)$ and $C_2 = (b,0,s)$ with $1-\alpha_0\le s\le r \le 1+\alpha_0$ and $b > 0$. With these choices, we have $d(C_1,C_2) = b + (r-s)$ and $\Delta(C_1,C_2) = |b-(r-s)|$. Consider the triangle $T$ in the plane whose vertices are $a$, $(0,0)$ and $(b,0)$. By elementary geometry, the angle $\phi$ at the vertex $a$ of $T$ is the same as $\angle u_1,u_2$. By the law of cosines,
\[
b^2 = r^2+s^2-2rs\cos\phi.
\]
Adding and subtracting $2rs$ and completing the square yields
\[
b^2 = (r-s)^2 + 2rs(1-\cos\phi).
\]
Note that by definition, $d(C_1,C_2)\Delta(C_1,C_2) = |b^2-(r-s)^2|$. Suppose $b > r-s$ (with a similar conclusion in the case $b \le r-s$), so rearranging, we have
\[
\frac{d(C_1,C_2)\Delta(C_1,C_2)}{rs} = 2(1-\cos\phi).
\]
Using the approximations $r,s\sim 1$ and $\cos\phi\sim 1-\frac{\phi^2}{2}$, we obtain
\[
d(C_1,C_2)\Delta(C_1,C_2)\sim \phi^2.
\]
Taking square roots yields the claim.
\end{proof}

\begin{prop}\label{prop:angle-em}
Suppose $\Omega^{(v)}\subset\overline\Omega^{(w)}$, where $\Omega^{(v)}$ is a $\delta,\tau$-rectangle, and $\overline\Omega^{(w)}$ is an $A\delta,A\tau$-rectangle. Let $a,\bar a$ be the center points of $\Omega^{(v)}$ and $\overline\Omega^{(w)}$, respectively, and let $e_m,\bar e_m$ be the positively oriented unit tangent vectors to $v,w$ at the points $a,\bar a$, respectively. Then $\angle e_m, \bar e_m \lesssim A^2\tau$.
\end{prop}
\begin{proof}
Let $d = d(v,w)$, $\Delta = \Delta(v,w)$, and let $\gamma,\overline\gamma$ denote the core arcs of $\Omega,\overline\Omega$. We make the same simplifying technical assumption that there exists a point $x\in \gamma\cap \overline\gamma$ to make use of Proposition \ref{prop:intersect-angle}. To remove this assumption, we note by replacing $v$ with a concentric circle of a slightly smaller or larger radius, we can arrange for $\gamma\cap\overline\gamma \ne\emptyset$, while keeping $\angle e_m,\bar e_m$ the same and $\Omega^{(v)}\subset\overline\Omega^{(w)}$.
By Proposition \ref{prop:intersect-angle}, the angle between the tangent lines to $v$ and $w$ at $x$, respectively, is $\sim \sqrt{d\Delta}$.

By the assumption $\Omega^{(v)}\subset\overline\Omega^{(w)}$, we have
\begin{equation}\label{eqn:big-intersection-contains-rect}
    |C_{A\delta,v}\cap C_{A\delta,w}|\ge |\Omega| \sim \delta\tau.
\end{equation}
On the other hand, by Lemma \ref{lemma:annuli-wolff},
\begin{equation}\label{eqn:meas-int-rects}
|C_{A\delta,v}\cap C_{A\delta,w}| \lesssim \frac{(A\delta)^2}{\sqrt{(d+A\delta)(\Delta+A\delta)}}.
\end{equation}
Combining inequalities \eqref{eqn:big-intersection-contains-rect} and \eqref{eqn:meas-int-rects} and using $\delta^{1/2}\le \tau$ gives
\[
\sqrt{d\Delta}\lesssim (\delta\tau)^{-1}(A\delta)^2 \le A^2\tau.
\]
Finally, because $\mathrm{dist}(a,x) + \mathrm{dist}(x,\bar a) \lesssim A\tau$, by comparing angles at $a$ and $\bar a$, we conclude $\angle e_m,\bar e_m = O(A^2)\tau + O(A)\tau = O(A^2)\tau$. 
\end{proof}

The following Proposition about engulfing rectangles is reminiscent of, but not the same as, Lemma 1.2 of Wolff's paper \cite{wolff2000local}, so we include the proof here for completion.

\begin{prop}[Engulfing]\label{prop:recip}
Let $1 < A,B \lessapprox 1$ and $\Omega^{(v)}$ be a $\delta,\tau$-rectangle contained in $C_{\delta,v}\cap C_{A\delta,w}$ for $v,w\in Q$. Suppose $\overline\Omega^{(w)}$ is an $A\delta,B\tau$-rectangle contained in $C_{A\delta,w}$ which contains $\Omega^{(v)}$. Then there exists $A_1\approx 1$ such that $\overline\Omega^{(w)}\subset C_{A_1AB^2\delta,v}\cap C_{A\delta,w}$.
\end{prop}
\begin{remark}
    In terms of taking tangency in the sense of Definition \ref{defn:tangency}, Proposition \ref{prop:recip} can be summarized succinctly as
    \[
    \Omega^{(v)}\subset\overline\Omega^{(w)}\implies v\in \mathbf D_{CAB^2\delta}(\overline\Omega^{(w)}),
    \]
    provided $C$ is sufficiently large. The threshold $CAB^2$ is probably not optimal, but the fact it is a polynomial in $A$ and $B$ is important.
\end{remark}
\begin{proof}
Let $d = d(v,w)$, $\Delta = \Delta(v,w)$, and let $\gamma,\overline\gamma$ denote the core arcs of $\Omega^{(v)},\overline\Omega^{(w)}$. We make the simplifying technical assumption that there exists a point $x\in \gamma\cap \overline\gamma\subset \Omega^{(v)}$. To remove this assumption, we note by replacing $v$ with a concentric circle of a slightly smaller or larger radius, we can arrange for $\gamma\cap\overline\gamma \ne\emptyset$, while keeping $\Omega^{(v)}\subset\overline\Omega^{(w)}$.

By translating, scaling by a factor $\sim 1$, and rotating our coordinate system if necessary, assume that $w = (0,0,1)$, $v = (a_1,a_2,s) =: (a,s)$, and that $x=1$ is on the positive real axis (see Figure \ref{fig:engulfing}).  In particular, we may assume that $\Omega^{(v)}$ is in our favorite position (see Remark \ref{rem:favorite-posn}), and its center is $e_1$.
\begin{figure}
    \centering
    \includegraphics[width=10cm]{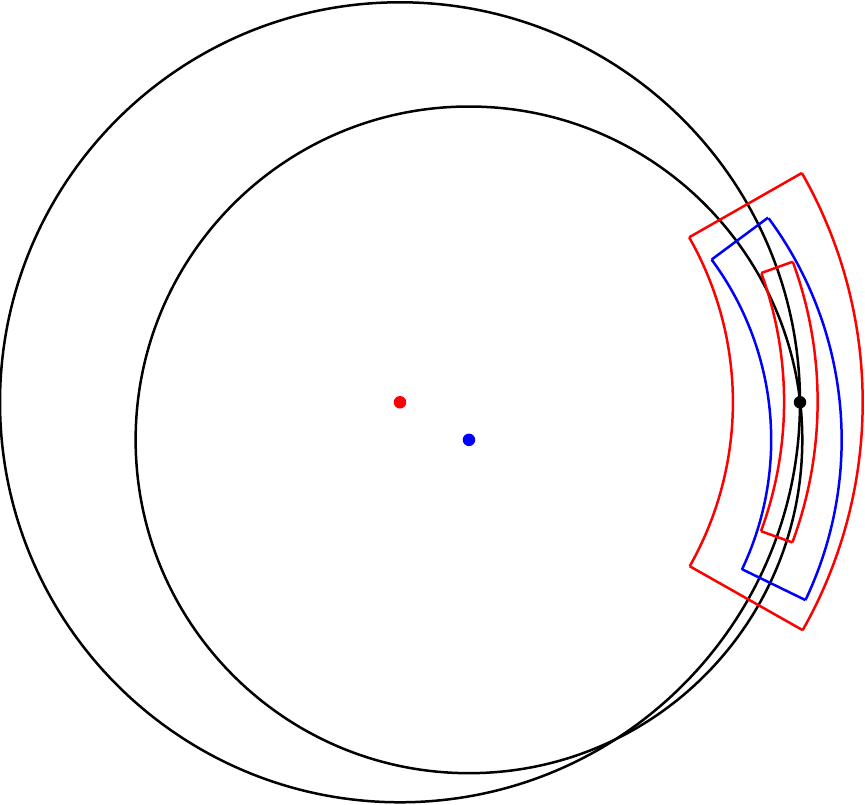}
    \caption{Engulfing rectangles}
    \label{fig:engulfing}
\end{figure}

With this choice of coordinate system, it suffices to show (using complex notation) that for $|r-1|<A\delta$ and $|\theta|\lesssim B\tau$, we have
\[
|re^{i\theta}-a| = s + O(AB^2)\delta,
\]
since in our chosen coordinate system, $\overline\Omega^{(w)}\subset\{re^{i\theta}:|r-1|<A\delta,|\theta|\lesssim B\tau\}$.

So assume $|r-1|<A\delta$ and $|\theta|\lesssim B\tau$. It follows by the triangle inequality we can replace $re^{i\theta}$ with $e^{i\theta}$ at the cost of $A\delta$. Next, because we assume $1 \in\gamma\cap\overline\gamma$, we have $s = |1-a|$, so we can substitute $|1-a|$ for $s$, and we are left with estimating
\[
||e^{i\theta}-a|-|1-a||.
\]
Because our circles lie in $Q = B(e_3,\alpha_0)$, we have the estimate $|e^{i\theta}-a|+|1-a|\sim 1$. Therefore multiplying $||e^{i\theta}-a|-|1-a||$ by $|e^{i\theta}-a|+|1-a|$, we only have to show
\[
|e^{i\theta}-a|^2-|1-a|^2 = O(AB^2)\delta.
\]
The upshot is we can use the trigonometric identity
\begin{align*}
|e^{i\theta}-a|^2 - |1-a|^2 &= 2\,\mathrm{Re}(a)(1-\cos\theta) - 2\,\mathrm{Im}(a)\sin\theta \\
&= O(\mathrm{Re}(a))\theta^2 + O(\mathrm{Im}(a))|\theta|\\
&\lesssim O(\mathrm{Re}(a))B^2\tau^2 + O(\mathrm{Im}(a))B\tau,
\end{align*}
and it suffices to estimate the components $a_1 = \mathrm{Re}(a)$ and $a_2 = \mathrm{Im}(a)$. We can use rectangle-lightplank duality to estimate both components simultaneously. We note $\Omega^{(v)}\subset C_{A\delta,w}$, so we have $w\in \mathbf D_{A\delta}(\Omega^{(v)})$, which is contained in an $\approx\delta\times \delta\tau^{-1}\times \delta\tau^{-2}$-lightplank, by Proposition \ref{prop:dual-to-d,t-rectangle}. By projecting this lightplank down to the plane $\R^2\times \{0\}$, we see that $|\mathrm{Re}(a)|\lessapprox \delta\tau^{-2}$, and $|\mathrm{Im}(a)|\lessapprox \delta\tau^{-1}$. Here we are using that $\Omega^{(v)}$ is in our favorite position (see Remark \ref{rem:favorite-posn}) so $\mathbf D_{A\delta}(\Omega^{(v)})\asymp P^{(o)}$, for a $\approx \delta\times\delta\tau^{-1}\times\delta\tau^{-2}$-lightplank containing $o = e_3$.

Collecting the estimates we have made so far, we have shown for arbitrary $|r-1|<A\delta$ and $|\theta|\lesssim B\tau$,
\[
||re^{i\theta}-a|-s|\lessapprox A\delta + \delta\tau^{-2}\cdot B^2\tau^2 + A\delta\tau^{-1}\cdot B\tau \lessapprox AB^2\delta.
\]
This finishes the proof.
\end{proof}

A useful fact about comparability is that that ``being comparable'' is almost a transitive relation. For the purpose of stating the next Proposition succinctly, if $\Omega,\Omega'$ are $A$-comparable, we write $\Omega\asymp_A\Omega'$. See  Proposition \ref{prop:both-tangent} on covering nearly lightlike separated pairs by lightplanks for our main application of this fact.
\begin{prop}[Being comparable is almost transitive]\label{almost-transitivity}
There is an absolute constant $C > 1$ such that if $\Omega^{(u)}\asymp_A\Omega^{(v)}$ and $\Omega^{(v)}\asymp_A\Omega^{(w)}$, then $\Omega^{(u)}\asymp_{CA^C}\Omega^{(w)}$.
\end{prop}
\begin{proof}
Recall the notation $u = (\bar u,u_3)$, $v = (\bar v,v_3)$, $w = (\bar w,w_3)$ for the points $u,v,w\in Q$, which we identify with circles $C_u,C_v,C_w\subset \R^2$ in the usual way. Consider the radial projection
\[
\pi\colon \mathbb R^2\setminus\{\bar v\}\to C_v,\quad z\mapsto \frac{z-\bar v}{|z-\bar v|}, \quad z\in \R^2\setminus \{\bar v\}
\]
onto $C_v$. By assumption, $\pi(\Omega^{(u)}\cup\Omega^{(w)})$ is contained in an arc of length $\sim A\tau$ containing the core arc of $\Omega^{(v)}$. Therefore, by Proposition \ref{prop:recip}, $\Omega^{(u)}\cup\Omega^{(w)}\subset C_{A^{O(1)}\delta,v}$, which finishes the proof.
\end{proof}

\section{Fourier transform of surface measure on the cone}\label{app:ft}

Here we record Lemma \ref{lem:fourier-decay} and its proof.
\begin{lemma}[Fourier transform estimate]
    Let $d\sigma$ be a smooth surface measure supported in $\mathbb Cone^2$. For any $\epsilon>0$ and any $N>1$, there is a constant $C(\epsilon,N)$ so that
    \[
    |\widecheck{d\sigma}(x)|\le C(\epsilon,N)\frac1{(1+|x|)^{\frac12-\epsilon}}\frac1{(1+d(x,\Gamma_0))^{N}}
    \]
    holds for all $x\in\R^3$, where $\Gamma_0$ is the lightcone $\{(a,r)\in\R^2\times\R:||a|-|r||=0\}$ with vertex $0$.
\end{lemma}
\begin{proof}
We will prove this by combining two estimates for $|\widecheck{d\sigma}(x)|$:
\begin{itemize}
\item[(i)] $|\widecheck{d\sigma}(x)|\lesssim (1+|x|)^{-\frac12+\epsilon}$
\item[(ii)] For every $N$, $|\widecheck{d\sigma}(x)|\lesssim_N(1+d(x,\Gamma_0))^{-N}$.
\end{itemize}
The conclusion follows by taking an appropriate geometric average of these two estimates. We may assume that $|x|\ge C$ for an appropriately large constant since $|\widecheck{d\sigma}(x)| \lesssim 1$ for $|x|\lesssim 1$.

We will start with (i). Suppose $|x|\sim r\gg 1$; our aim is to show $|\widecheck{d\sigma}(x)|\lesssim r^{-\frac12+\epsilon}$. We divide $\mathbb Cone^2$ into $\sim r^{\frac12-\epsilon}$-many strips $\theta$ of angular width $r^{-\frac12+\epsilon}$ and let $\{\eta_\theta\}$ be a smooth partition of unity subordinate to $\{\theta\}$. Then with $d\sigma_\theta = \eta_\theta \,d\sigma$,
\[
\widecheck{d\sigma}(x) = \sum_\theta\widecheck{d\sigma}_\theta(x).
\]
For each $\theta$, we let $\theta^*$ be the lightplank containing the origin of dimensions $1\times r^{\frac12-\epsilon}\times r^{1-2\epsilon}$ dual to the $r^{-1+2\epsilon}$-neighborhood of $\theta$. By the Schwartz decay of $\widecheck{d\sigma}_\theta(x)$, we have
\[
|\widecheck{d\sigma}_\theta(x)|\lesssim_N |\theta|\sum_{j=0}^\infty 2^{-jN}1_{2^j\theta^*}(x).
\]
Since we assume $|x|\sim r\gg r^{1-2\epsilon}$, and the directions of $\theta^*$ are $r^{-\frac12+\epsilon}$-separated, $x$ lies in at most $\lessapprox 1$ of the $\theta^*$. Therefore,
\[
|\widecheck{d\sigma}_\theta(x)|\lessapprox |\theta| \sim r^{-\frac12+\epsilon}.
\]

Now we prove (ii), but instead of using wave packets, we give a proof based on stationary phase considerations. For this proof, we use the notation $(x',x_3)\in \R^2\times \R$ for points in $\R^3$ instead of the notation $(\bar x,x_3)$ we used earlier. Let $x = (x',x_3)$ with $|x|\gg 1$. By the symmetry $(x',x_3)\mapsto (x',-x_3)$, we may assume that $x_3>0$. We also assume that $|x'|>x_3>0$; the case $|x'|<x_3$ is similar. For an appropriate smooth and compactly supported function $a(\xi)$ in $\{\xi\in\R^2:1<|\xi|<2\}$, we  write
\[
\widecheck{d\sigma}(x) = Ea(x) =  \int a(\xi)e^{2\pi i(x'\cdot \xi+x_3|\xi|)}\,d\xi.
\]
Here $E$ is the Fourier extension operator for the cone.

Let $w$ be the nearest point on the cone $\Gamma_0$ to $x$. By elementary geometry, $x-w$ is orthogonal to the lightcone at $w$, and from this, we can compute the coordinates of $w$ in terms of $x$:
\[
w = w(x) = \left(\frac{|x'|+x_3}{2}\frac{x'}{|x'|},\frac{|x'|+x_3}{2}\right).
\]
Note from this formula for $w$ that
\[
d(x,\Gamma_0)= |x-w| \le |x'|-x_3 \lesssim |x-w| = d(x,\Gamma_0),
\]
so $d(x,\Gamma_0)\sim||x'|-x_3|$. Write
\begin{align*}
Ea(x) = Ea(w+ (x-w)) &=  \int a(\xi) e^{2\pi i(w'\cdot\xi + w_3|\xi|)}e^{2\pi i[(x'-w')\cdot \xi + (x_3-w_3)|\xi|]}\,d\xi \\
&= \int a(\xi) e^{2\pi i(w'\cdot\xi + w_3|\xi|)} e^{2\pi i\frac{|x'|-x_3}{2} (\frac{x'}{|x'|}\cdot \xi - |\xi|)}   \\
&= \int a(\xi) e^{2\pi i\phi_1(\xi)}e^{2\pi i\lambda\phi_2(\xi)}\,d\xi,
\end{align*}
where
\begin{align*}
\phi_1(\xi) &= w'\cdot\xi + w_3|\xi|\\
\phi_2(\xi) &= \frac{x'}{|x'|}\cdot\xi - |\xi| \\
\lambda = \lambda(x) &= \frac{|x'|-x_3}2.
\end{align*}
Let $\Sigma_1 = \{1<|\xi|<2:\nabla \phi_1(\xi) = 0\}$, and similarly denote $\Sigma_2=\{1<|\xi|<2:\nabla\phi_2(x)=0\}$.
Since $w$ is the nearest point to $x$ lying in $\Gamma_0$, the critical points of $\phi_1(\xi)$ in $\{1<|\xi|<2\}$ are contained in the line segment
\[
\Sigma_1 = \{1<|\xi|<2: \frac{\xi}{|\xi|} = -\frac{x'}{|x'|}\}.
\]
Likewise, $\Sigma_2$ is contained in the line segment
\[
\Sigma_2 = \{1<|\xi|<2:\frac{\xi}{|\xi|} = \frac{x'}{|x'|}\}.
\]
Consider the open sets
\[
U_1 = \{1<|\xi|<2:\angle(\xi,-x')>\frac{1}{5}\},\quad U_2 = \{1<|\xi|<2:\angle(\xi,-x') < \frac25\},
\]
and a smooth partition of unity $\{\eta_1,\eta_2\}$ subordinate to $\{U_1,U_2\}$. By linearity, $Ea = E(a\eta_1) + E(a\eta_2)$. Since the phase $x'\cdot\xi + x_3|\xi|$ has no critical points in $U_1$, we have $|E(a\eta_1)(x)|\lesssim_N (1+|x|)^{-N}\le d(x,\Gamma_0)^{-N}$ by integrating by parts. Therefore, it suffices to show that $|E(a\eta_2)(x)|\lesssim_N d(x,\Gamma_0)^{-N}$.

Since we only work with $a\eta_2$ from now on, to reduce clutter, we let $a_0$ denote $a\eta_2$. By definition of $\Sigma_2$,
\begin{equation}\label{eqn:good-supp}
\mathrm{dist}(\supp a_0, \Sigma_2)\ge\frac{1}{100}.
\end{equation}
Lastly, we note that the phase $\phi_2$ satisfies
\begin{equation}\label{eqn:harmless-phase}
\sup\{|\partial^\alpha\phi_2(\xi)|:1<|\xi|<2,|\alpha|\le N\}\le C_N
\end{equation}
for each $N\ge 1$.
Consider the following vector field and its transpose
\[
L = \frac{1}{2\pi i\lambda}\mathbf v\cdot \nabla,\quad L^tf = -\frac{1}{2\pi i\lambda}\nabla\cdot(f\mathbf v)
\]
where $\mathbf v(\xi) = \nabla\phi_2(\xi)/|\nabla\phi_2(\xi)|$, which is well defined and smooth throughout $\supp a_0$ by Equation \eqref{eqn:good-supp}. By definition $Le^{2\pi i\lambda\phi_2} = e^{2\pi i\lambda\phi_2}$, and consequently integrating by parts one time,
\begin{align*}
Ea_0( x) &= \int L^t(a_0 e^{2\pi i\phi_1})e^{2\pi i\lambda\phi_2}\\
&= -\frac{1}{2\pi i\lambda}\int \nabla\cdot(a_0e^{2\pi i\phi_1}\mathbf v)e^{2\pi i\lambda\phi_2}.
\end{align*}
Using the vector calculus identity
\[
\nabla\cdot (fg\mathbf v) = f\nabla g\cdot \mathbf v + g\nabla f\cdot \mathbf v + fg\nabla\cdot \mathbf v,
\]
applied with $f = a_0$ and $g = e^{2\pi i\phi_1}$, we get
\begin{align*}
Ea_0(x) = -\frac{1}{2\pi i\lambda}\bigg(\int a_0 [2\pi i\,e^{2\pi i\phi_1}\nabla\phi_1\cdot\mathbf v]e^{2\pi i\lambda\phi_2} + \int e^{2\pi i\phi_1}\underbrace{(\nabla a_0\cdot\nabla \mathbf v+a_0\nabla\cdot \mathbf v)}_{\equiv\ a_1}e^{2\pi i\lambda\phi_2}\bigg).
\end{align*}
Note that
\[
\nabla \phi_1(\xi) = w'+w_3\frac\xi{|\xi|} = (\frac{|x'|+x_3}{2})(\frac{x'}{|x'|}+\frac{\xi}{|\xi|})
\]
and
\[
\mathbf v = \frac{1}{|\nabla \phi_2|}(\frac{x'}{|x'|}-\frac{\xi}{|\xi|}).
\]
Therefore, $\nabla\phi_1\cdot\mathbf v = 0$, so $Ea_0(x)$ simplifies to
\[
Ea_0(x) = -\frac{1}{2\pi i\lambda}\int a_1 e^{2\pi i\phi_1}e^{2\pi i\lambda\phi_2}.
\]
Repeating the same integration-by-parts argument $N$ times with $a_1$ in place of $a_0$, and then $a_2$ in place of $a_1$, and so on, where
\[
a_{j+1} = \nabla a_j\cdot \nabla \mathbf v + a_j\nabla\cdot\mathbf v, \quad j \ge 1,
\]
we get
\[
|Ea_0(x)| \le \frac{C_N}{\lambda^N}\int |a_N|.
\]
By Equations \eqref{eqn:good-supp}, \eqref{eqn:harmless-phase} and the definition of $\mathbf v$, this is bounded by $C_N'/\lambda^N \sim_N \frac{1}{(|x'|-x_3)^N}$. Finally, since $|x'|-x_3 \sim d(x,\Gamma_0)$, we have proved
\[
|Ea(x)| \lesssim_N  \frac{1}{d(x,\Gamma_0)^N}.
\]
Together with the proof of (i), this finishes the proof of the lemma.
\end{proof}

\printbibliography

\end{document}